\documentclass[11pt,oneside]{book}
\usepackage{amsthm,amsmath,a4,amssymb,enumerate,verbatim}
\usepackage[latin1]{inputenc}  
\usepackage[french,english]{babel}

\def\@chapter[#1]#2{\ifnum \c@secnumdepth >\m@ne
                       \if@mainmatter
                         \refstepcounter{chapter}%
                         \typeout{\@chapapp\space\thechapter.}%
                         \addcontentsline{toc}{chapter}%
                                   {\protect\numberline{\thechapter}#2}%
                       \else
                         \addcontentsline{toc}{chapter}{#2}%
                       \fi
                    \else
                      \addcontentsline{toc}{chapter}{#2}%
                    \fi
                    \chaptermark{#1}%
                    \addtocontents{lof}{\protect\addvspace{10\p@}}%
                    \addtocontents{lot}{\protect\addvspace{10\p@}}%
                    \if@twocolumn
                      \@topnewpage[\@makechapterhead{#2}]%
                    \else
                      \@makechapterhead{#2}%
                      \@afterheading
                    \fi}
                    
\newtheorem{thm}{Théorème}[chapter]
\newtheorem{cor}[thm]{Corolaire}
\newtheorem{lem}[thm]{Lemme}
\newtheorem{prop}[thm]{Proposition}
\theoremstyle{definition}
\newtheorem{defn}[thm]{Définition}
\theoremstyle{remark}
\newtheorem{rem}[thm]{Remarque}
\newtheorem{exe}[thm]{Exemple}
\newtheorem{spec}[thm]{Spéculation}
\numberwithin{equation}{section}

\newcommand\balank[1]{\rule[+.7ex]{#1}{.4pt}}

\newcommand{\Hom}{\textnormal{Hom}}

\newcommand{\Aut}{\textnormal{Aut}}

\newcommand{\Ker}{\textnormal{Ker}}

\newcommand{\SL}{\textnormal{SL}}
\newcommand{\GL}{\textnormal{GL}}

\newcommand{\lp}{(\!(}
\newcommand{\rp}{)\!)}

\newcommand{\Precone}{\textnormal{Precone}}
\newcommand{\Cone}{\textnormal{Cone}}

\newcommand{\K}{\mathbf{K}}
\newcommand{\Z}{\mathbf{Z}}
\newcommand{\R}{\mathbf{R}}
\newcommand{\Q}{\mathbf{Q}}
\newcommand{\F}{\mathbf{F}}
\newcommand{\T}{\mathbf{T}}
\newcommand{\N}{\mathbf{N}}
\def\threedots{\hbox{\,.\hspace{.13em}.\hspace{.13em}.\;}}

\begin{document}


\title{  {\normalsize Université Paris-Sud, Laboratoire de Mathématiques} \\ ~\\  Mémoire d'habilitation à diriger des recherches \\ ~ \\~\\~\\ {\huge {\bf Aspects de la géométrie des groupes}} \\ ~ \\~\\  par {\huge Yves de Cornulier} \\ {\normalsize (CNRS -- Université Paris-Sud 11)}\\ ~\\ ~ \\ {\large soutenue le 15 mai 2014 devant le jury composé de}  \\\begin{flushleft} \hspace{4cm}{\Large Yves Benoist} \\ \hspace{4cm}{\Large Thomas Delzant (rapporteur)}\\ \hspace{4cm}{\Large Damien Gaboriau }\\ \hspace{4cm}{\Large Pierre Pansu (rapporteur interne)} \\ \hspace{4cm}{\Large Bertrand Rémy}\\ \hspace{4cm}{\Large Alain Valette}\\ \hspace{4cm}{\Large (Rapporteur excusé: Shmuel Weinberger)}\end{flushleft} }
\date{}


\maketitle

\selectlanguage{french}

\tableofcontents

\chapter*{Avant-propos}

Le mémoire qui suit est constitué de cinq chapitres, chacun étant une introduction (en français) à quelques uns des articles écrits entre ma thèse et maintenant. 

Les trois premiers chapitres, qui constituent le coeur de ce mémoire, relèvent de la théorie géométrique des groupes, et plus précisément de la géométrie à grande échelle des groupes, notamment des groupes de Lie (chapitres 1 et 2) et des groupes hyperboliques localement compacts (chapitre 3). 

Le quatrième chapitre relève de la théorie des groupes via leurs représentations unitaires et actions isométriques sur des espaces de Hilbert. 

Enfin le cinquième chapitre relève plus de la théorie ``structurelle" des groupes, en étudiant, pour un groupe donné (discret ou localement compact), l'espace de ses sous-groupes fermés (ou sous-groupes distingués fermés).

\medskip

Je remercie les rapporteurs de leur relecture, ainsi que tous les membres du jury pour avoir accepté d'y participer. Je remercie tous les collaborateurs et autres collègues avec qui j'ai eu l'occasion de discuter et d'échanger. Depuis mon entrée au CNRS en octobre 2006, j'ai passé un peu plus de 4 ans à l'IRMAR à Rennes, et maintenant 3 ans au LMO à Orsay. Je remercie chaleureusement ces laboratoires pour leur accueil.


\chapter{Cônes asymptotiques de groupes}

Ce chapitre est une synthèse des articles suivants 
\begin{itemize}
\item \cite{CJT} {\em Dimension of asymptotic cones of Lie groups}, J. Topology, 19 pages (2008),
\item \cite{CI} {\em Asymptotic cones of Lie groups and cone equivalences}, Illinois J. Maths, 23 pages (2011),
\end{itemize}
et quelques corolaires inédits.

Les articles \cite{CJT,CI} forment un ensemble qui aurait pu constituer un seul article: en effet une partie de \cite{CI} permet d'apporter un éclairage conceptuel au principal point technique de \cite{CJT}. En résumé, l'article \cite{CI} introduit les notions de fonction (entre espaces métriques) cône-définie et cône-lipschitzienne; voir également plus bas. Cette notion est plus générale que celle de fonction lipschitzienne à grande échelle, et apparaît naturellement dans le contexte des groupes de Lie et de l'étude de leurs cônes asymptotiques. 


On rappelle que si $\omega$ est un ultrafiltre non principal sur l'ensemble des entiers et si $X$ est un espace métrique, on note $\Precone(X)$ l'ensemble des suites dans $X$ à croissance au plus linéaire ($(x_n)$ telle que $d(x_0,x_n)=O(n)$); on le munit de la pseudo-distance $d_\omega((x_n),(y_n))=\lim_\omega d(x_n,y_n))/n$ et $\Cone_\omega(X)$ est l'espace métrique obtenu en identifiant dans $(\Precone(X),d_\omega)$ les points à distance nulle, on l'appelle cône asymptotique de $X$ (par rapport à l'ultrafiltre $\omega$); l'image de la suite $(x_n)$ y est notée $(x_n)^\omega$.

Cette notion très générale est due à Van der Dries et Wilkie \cite{DW} et permettait de placer dans un cadre général du point de vue de la théorie des modèles une définition beaucoup plus restrictive (convergence de Gromov-Hausdorff) introduite par Gromov \cite{Gro81} pour caractériser les groupes à croissance polynomiale. Cela a apporté un éclairage important dans le cas Gromov-hyperbolique, où les cônes asymptotiques sont des arbres réels (Gromov, Paulin, Kapovich-Leeb). Pour des groupes plus généraux, c'est Gromov lui-même qui a démontré la pertinence des cônes asymptotiques (hors convergence Gromov-Hausdorff et cas hyperbolique) en géométrie, dans son livre fondateur \cite{Gro93}, qui a engendré de nombreux travaux sur la géométrie des groupes, au premier plan desquels figurent les travaux de Drutu, Kleiner, Leeb, etc. Un des intérêts de cette notion est qu'il s'agit d'un invariant topologique (et même métrique à homéomorphisme bilipschitzien près) de la classe de quasi-isométrie de $G$.

Un des points du propos de Gromov dans \cite{Gro93} était que bien que les cônes asymptotiques des groupes à croissance non polynomiale paraissent à première vue ``énormes" (non localement compacts, non séparables\dots), ils sont, dans beaucoup de cas, très raisonnables (par exemple, de dimension topologique finie). Ainsi, dans le cas des groupes de Lie semi-simples, ils ont une structure d'immeuble affine, et la dimension du cône asymptotique est alors égale au rang réel du groupe. Dans le cas général des groupes de Lie connexes, le calcul de cette dimension est le résultat principal de \cite{CJT}.

\begin{thm}\label{codimre}Soit $G$ un groupe de Lie connexe (muni d'une métrique riemannienne invariante à gauche). Alors la dimension topologique du cône asymptotique de $G$ est un nombre entier explicite, qui, lorsque $G$ est un groupe de Lie résoluble simplement connexe, est égal à la codimension du radical exponentiel de $G$.
\end{thm}

Ici, le radical exponentiel $\text{Exp}(G)$ de $G$ (introduit discrètement par Guivarc'h puis redécouvert par Osin) est l'ensemble des éléments exponentiellement distordus de $G$, de sorte que sa codimension est égale à la dimension de $G/\text{Exp}(G)$, qui est le plus gros quotient de $G$ à croissance polynomiale.

Dans le cas général, la formule \cite[Corollary 1.6]{CJT} fait intervenir trois contributions: le rang de la partie semi-simple, le rang du centre (discret) de la partie semi-simple, et un troisième terme qui donne la contribution du radical. Cependant, le cas général se ramène au cas résoluble, et même au cas résoluble {\em triangulable}, en vertu du lemme ci-dessous.

\begin{defn}\label{glt}On appelle groupe de Lie triangulable un groupe de Lie isomorphe à un sous-groupe connexe fermé de matrices réelles triangulaires supérieures, ou de manière équivalente, dont la représentation adjointe est triangulable sur $\mathbf{R}$.
\end{defn}

\begin{lem}\label{lemtr}
Soit $G$ un groupe de Lie connexe. Alors il existe des morphismes continus de groupes localement compacts $G\hookleftarrow G_1\twoheadrightarrow G_2 \hookrightarrow G_3\hookleftarrow G_4$, où $G_4$ est un groupe de Lie triangulable, où les flèches $\hookrightarrow$ sont des inclusions cocompactes et la flèche $\twoheadrightarrow$ est le quotient par un sous-groupe compact distingué.
\end{lem}

Ainsi, les flèches ci-dessus sont des quasi-isométries et induisent donc des homéo\-mor\-phismes bilipschitziens entre les cônes asymptotiques. La preuve (constructive) de ce lemme, qui ramène tout le problème difficile de la classification à quasi-isométrie près des groupes de Lie aux groupes de Lie connexes à ceux qui sont triangulables, est dispersée (maladroitement) dans \cite{CJT} (Lemma 6.7 et~2.4). 


Concentrons-nous maintenant sur le cas d'un groupe de Lie triangulable $G$, de radical exponentiel $E$.
Muni de la distance $d_G$ induite par $G$, le cône asymptotique de $E$ est de dimension 0. Un argument topologique dû à Gromov \cite{Gro93}, et précisé par Burillo \cite{Bur} ainsi que dans \cite{CJT} permet alors de majorer la dimension topologique de $\Cone_\omega(G)$ par celle de $\Cone_\omega(G/E)$, qui est connue pour être égale à $\dim(G/E)$ grâce au travail de Pansu \cite{Pan1}: on peut voir en effet $\Cone(E,d_G)$ comme la fibre de la projection canonique $\Cone_\omega(G)\to\Cone_\omega(G/E)$ (qui n'est pas un fibré localement trivial!). 


La difficulté est de démontrer que la dimension topologique de $\Cone_\omega(G)$ est supérieure ou égale à $\dim(G/E)$. Dans les cas les plus simples (par exemple si $E$ est abélien), on a une décomposition en produit semi-direct $G=E\rtimes N$, de sorte que $\Cone_\omega(G)$ contienne une copie isométrique de $\Cone_\omega(N)$, qui est de dimension $N$. En général, il n'existe pas de telle décomposition semi-directe; des contre-exemples apparaissent dans \cite{CJT} (le plus petit contre-exemple a $\dim(G)=5$ et $\dim(G/E)=2$). Le travail dans \cite{CJT} consiste alors à fabriquer explicitement un plongement de $\Cone_\omega(G/E)$ dans $\Cone_\omega(G)$. Le point de vue de \cite{CI} permettant de montrer un résultat plus précis, j'adapte  partir de maintenant ce point de vue.

Il est connu depuis le départ que toute application lipschitzienne à grande échelle (i.e.\ satisfaisant une inégalité du type $d(f(x),f(y))\le Cd(x,y)+C'$ induit (fonctoriellement) une application lipschitzienne (avec la même constante $C$) entre cônes asymptotiques: en effet, l'application entre suites $(x_n)\mapsto (f(x_n))$ envoie précone sur précone et est compatible avec la relation d'équivalence définissant le cône. Or \cite{CI} part de l'observation qu'on peut définir avec la même formule une fonction entre les cônes, sous une hypothèse plus faible sur $f$; par exemple, une fonction à croissance sous-linéaire induit une fonction constante. 

\begin{defn}[\cite{CI}]
Soient $X_1,X_2$ des espaces métriques et $f:X_1\to X_2$ une fonction. On dit que 
\begin{itemize}
\item $f$ est cône-définie si $|f(x)|=O(|x|)$ et $(d(x,y)+1)/(|x|+|y|)$ tend vers zéro implique $d(f(x),f(y))/(|x|+|y|)$ tend vers zéro.
\item $f$ est cône-$c$-lipschitzienne si $d(f(x),f(y))\le cd(x,y)+o(|x|+|y|)$
\end{itemize}
\end{defn}
Ici $|\cdot|$ désigne la distance à un point-base (dont le choix est sans importance) et $o(|x|+|y|)$ signifie un terme $\le q(|x|+|y|)$ où $q:\R_+\to\R_+$ est une fonction bornée et tendant vers 0 à l'infini.

Dans \cite{CI}, on montre que la condition cône-définie est la condition optimale pour pouvoir définir pour tout ultrafiltre $\omega$, la fonction $\Cone_\omega(X_1)\to\Cone_\omega(X_2)$ par $(x_n)^\omega\mapsto (f(x_n))^\omega$. On vérifie de plus qu'alors cette fonction est automatiquement continue, et, en outre, qu'elle est $c$-lipschitzienne pour tout $\omega$ si et seulement si $f$ est cône-$c$-lipschitzienne. 

Cela permet de définir deux catégories, dont les objets sont les espaces mét\-riques et les flèches sont les fonctions cône-définies (resp. cône-lipschitzienne) modulo cône-équivalence ($f\sim f'$ si $d(f(x),f'(x))=o(|x|)$). Les fonctions induisant des isomorphismes de cette catégorie sont appelées cône-équivalences, resp.\ cône-bilipschitz-équivalences. Elles induisent donc des homéomorphismes (resp.\ homéo\-mor\-phismes bilipschitziens) entre cônes asymptotiques. Le Theorem 1.2 de \cite{CI} (qui conceptualise l'argument de \cite{CJT}) peut s'exprimer (de façon légèrement imprécise) comme ceci:

\begin{thm}\label{t_cel}
Soit $G$ un groupe de Lie triangulable, $E$ son radical exponentiel, et $N=G/E$. Alors on peut associer à $G$, de façon naturelle, un produit semi-direct $G'=E\rtimes N$ dont $E$ est le radical exponentiel, muni d'une cône-bilipschitz-équivalence $G\to G'$, et tel que l'action de $N$ sur $E$ est $\R$-diagonalisable (en particulier $[N,N]$ est central dans $G'$).
\end{thm}
La construction est naturelle et fait que l'action de $N$ sur $E$ est, intuitivement, la partie diagonale de l'action de $G/N$ sur $E$\threedots  à ceci près que cette action de $G/N$ sur $E$ n'existe pas (sauf si $E$ est abélien)! plus précisément, la partie diagonale de l'action de $G$ sur $E$ factorise par $N$, et c'est cette dernière qui permet de définir le groupe $G'$. En pratique, on définit $G'$ en modifiant la loi de $G$, si bien que la cône-équivalence du théorème est l'application identité entre $G$ et $G$ muni d'une loi modifiée.

Le théorème \ref{t_cel} ramène donc le calcul de la dimension du cône asymptotique d'un groupe de Lie triangulable au cas où le radical exponentiel est scindé et on a mentionné plus haut que dans ce cas, la borne inférieure sur la dimension du cône asympotique se ramène au cas nilpotent dû à Pansu.

Par un résultat de Pauls \cite{Pauls}, un groupe de Lie connexe nilpotent non abélien n'admet aucun plongement bilipschitzien dans un espace métrique CAT(0) (CAT(0) est une forme métrique de la condition d'être à courbure négative ou nulle). La construction dans le cône asymptotique de $G$ d'une copie bilipschitzienne du cône de $G/E$ permet de montrer le résultat suivant (qui n'invoque pas, dans son énoncé, les cônes asymptotiques!).

\begin{thm}[{\cite[Theorem 1.10]{CJT}}]Soit $G$ un groupe de Lie triangulable et $E$ son radical exponentiel. Équivalences:
\begin{enumerate}[(i)]
\item\label{gqic} $G$ se plonge quasi-isométriquement dans un espace CAT(0);
\item\label{gea} $G/E$ est abélien.
\end{enumerate} 
\end{thm}

La partie positive ((\ref{gea})$\Rightarrow$(\ref{gqic})) de ce résultat fonctionne comme ceci: on considère un plongement de $G$ comme sous-groupe fermé de $\SL_d(\R)$, et on considère le plongement diagonal $G\hookrightarrow \SL_d(\R)\times (G/E)$. Alors ce plongement est quasi-isométrique. Le groupe de droite agit naturellement sur un espace symétrique à courbure négative ou nulle (le produit riemannien de l'espace symétrique de $\SL_d(\R)$ par l'espace euclidien $G/E$), et l'action de $G$ sur ce dernier définit le plongement désiré en considérant une application orbitale.

Une variante de l'argument, basé sur les mêmes calculs, montre la chose suivante:

\begin{lem}\label{tsqi}
On considère un groupe de Lie triangulable $G'=E\rtimes N$ de radical exponentiel $E$, tel que l'action de $N$ sur l'algèbre de Lie de $E$ est $\R$-diagonalisable (comme dans la conclusion du théorème \ref{t_cel}). Étant donné un plongement de $G/[N,N]$ comme sous-groupe fermé de $\SL_d(\R)$, le plongement diagonal $G\hookrightarrow \SL_d(\R)\times G/E$ est un plongement quasi-isométrique.\qed
\end{lem}

On peut donc interpréter $G$ comme un produit fibré de $G/[N,N]$ et $N=G/E$ au dessus de leur projections sur $G/[G,G]=N/[N,N]$, et ce en un sens métrique. Ceci implique une telle décomposition au niveau du cône asymptotique:

\begin{cor}
Soit $G$ comme dans le lemme \ref{tsqi}. Pour tout $\omega$, le cône asymptotique $\Cone_\omega(G)$ est naturellement le produit fibré des surjections lipschitziennes $\Cone_\omega(G/[N,N])\twoheadrightarrow \Cone_\omega(G/[G,G])$ et $\Cone_\omega(G/E)\twoheadrightarrow \Cone_\omega(G/[G,G])$. 
\end{cor}

L'intérêt de ce corolaire est que les objets définissant ce produit fibré sont, en un sens convenable, explicitables à homéomorphisme bilipschitzien près:
\begin{itemize}
\item  $\Cone_\omega(G/E)$ est un groupe de Lie nilpotent Carnot-graduable simplement conn\-exe muni d'une métrique sous-riemannienne (de Carnot-Carathéodory) \cite{Pan1} et la fonction vers $\Cone_\omega(G/[G,G])$ est sa projection canonique vers son abélia\-nisé;
\item si $H=G/[N,N]$, alors $\Cone_\omega(H)$ peut s'interpréter comme $H(\R_\omega)/H(\mathbf{A}_\omega)$, où $\R_\omega$ est le corps de Robinson de $\R$ relatif à $\omega$: c'est l'ensemble des suites $(x_n)$ de réels à croissance au plus exponentielle, modulo celles à décroissance surexponentielle relativement à $\omega$. C'est un corps valué sur $\R$, par le taux de croissance exponentiel relativement à $\omega$, et $\mathbf{A}_\omega$ désigne le sous-anneau des éléments de valuation positive ou nulle. Le groupe $H$ peut être vu comme pseudo-algébrique (i.e., défini à l'intérieur d'un groupe algébrique réel triangulable $L=DU$ comme image réciproque d'un sous-espace vectoriel de $D(\R)$), ce qui permet de définir $H(K)$ pour toute extension réelle-close de $\R$ sur qui on peut définir l'exponentiation d'un nombre positif par un nombre réel. La distance sur $H(\R)$ (définie à équivalence quasi-isométrique près) induit une pseudo-distance sur $H(\R_\omega)$ définie uniquement en termes de $H$ et de la valuation de $\R_\omega$, dont l'ensemble des éléments à pseudo-distance nulle de l'identité est précisément $H(\mathbf{A}_\omega)$. En particulier, l'espace métrique $\Cone_\omega(H)$ ne dépend, à homéomorphisme bilipschitzien près, que de $H$ et du corps valué $\R_\omega$, c'est aussi le cas de l'application vers l'abélianisé. 
\end{itemize}

Or on sait (Thornton) que si (et seulement si) l'hypothèse du continu est vraie, alors le corps valué $\R_\omega$, à isomorphisme de corps valué près, ne dépend pas de $\omega$. On déduit:

\begin{cor}\label{cup}
Soit $G$ un groupe de Lie connexe. Si l'hypothèse du continu est vraie, alors modulo homéomorphisme bilipschitzien, l'espace métrique $\Cone_\omega(G)$ ne dépend pas de l'ultrafiltre non principal $\omega$.
\end{cor}

Ce résultat était connu dans le cas où $G$ est semi-simple \cite{KSTT}, où il est établi que si $G$ est absolument simple réel alors réciproquement la négation de l'hypothèse du continu implique l'existence de plusieurs (en fait $2^{2^{\aleph_0}}$) ultrafiltres donnant des cônes asymptotiques non homéomorphes.

\begin{spec}Il est important d'interpréter cette réciproque de manière correcte. Une interprétation erronée mais répandue est de dire que la géométrie du groupe de Lie dépend du modèle de ZFC dans lequel on se place (ce résultat étant mis en parallèle avec des résultats de dépendance en $\omega$ pour certains groupes de type fini comme dans \cite{TV}, dont la géométrie varie explicitement avec l'échelle). Mon interprétation est qu'en l'absence de l'hypothèse du continu, ZFC est trop faible pour montrer l'existence d'un homéomorphisme entre les cônes asymptotiques, bien qu'ils aient les mêmes propriétés en un sens précis modèle-théorique (voir \cite[\S 5]{KSTT}), mais que cependant le corolaire indique que les cônes ont \og la même forme\fg. Il serait intéressant de préciser cela.
\end{spec}

Terminons par un autre résultat de \cite{CJT}, dissimulé dans \cite[\S 9]{CJT}. Dans certains groupes de Lie connexes, on peut expliciter le cône asymptotique sans référence au corps de Robinson, ce qui donne un résultat d'indépendance (à homéomorphisme bilipschitzien près) par rapport à l'ultrafiltre sans référence à l'hypothèse du continu. C'est le cas 
\begin{itemize}
\item dans le cas nilpotent déjà mentionné (Pansu \cite{Pan1});
\item dans le cas hyperbolique: en effet, le cône asymptotique pour tout ultrafiltre non principal est alors un arbre réel complet; de plus il est homogène et, cas élémentaire mis à part, sa valence au point-base est $2^{\aleph_0}$ (car le bord a $2^{\aleph_0}$ éléments); \cite[Theorem 3.5]{MNO} implique alors que tout cône asymptotique est isométrique à l'arbre réel universel $\T=T_{2^{\aleph_0}}$ (défini à isométrie près par le fait d'être non vide, complet, et de valence $2^{\aleph_0}$ en tout point. Ce résultat sera redécouvert par la suite dans \cite{EP}.
\end{itemize}

\begin{prop}
Soit $G$ un groupe de Lie connexe abélien-par-nilpotent. 
Alors on peut expliciter le cône asymptotique de $G$ (à homéomorphisme bilipschitzien près), de manière qui ne dépend pas de l'ultrafiltre non principal. 
\end{prop}
\begin{proof}[Explicitation:] (le lecteur peut commencer par regarder l'exemple \ref{exv}). 
On procède comme ceci: on peut supposer que $G$ est triangulable (sinon on s'y ramène à l'aide du lemme \ref{lemtr}). À l'aide d'un sous-groupe de Cartan (au sens de Bourbaki), on peut écrire $G=E\rtimes N$ avec $E$ son radical exponentiel (qui est abélien), et $N$ nilpotent. On peut aussi supposer, grâce au théorème \ref{t_cel}, que $N$, agit de façon diagonalisable sur $E$. On décompose $E$ en sous-espaces propres communs aux éléments de $N$, ce qui donne $E=\bigoplus E_\alpha$, où $\alpha$ parcourt les morphismes non nuls de $N$ dans $\R$, si bien que $E_\alpha$ est l'ensemble des éléments de $E$ sur qui tout $g\in N$ agit par multiplication par $e^{\alpha(g)}$. Soit $W$ l'ensemble (fini) des $\alpha$ tel que $E_\alpha\neq 0$, dit ensemble des poids. Pour $\alpha\in W$, soit $G_\alpha=G/(\Ker(\alpha)\bigoplus_{\beta\neq\alpha}E_\beta)\simeq E_\alpha\rtimes (N/\Ker(\alpha))$.


Commençons par étudier le cas $N$ agit fidèlement sur $E$ (en notant que cela n'est pas le cas si $N$ est non abélien!). Par un argument banal, le morphisme diagonal $\iota$ de $G$ dans $\prod_\alpha G_\alpha$ est alors un plongement quasi-isométrique.
Chaque $G_\alpha$ est hyperbolique non élémentaire et donc son cône asymptotique est bilipschitz homéomorphe à une copie $T_\alpha$ de l'arbre réel universel $\T$, la projection $p_\alpha$ de $G_\alpha$ vers $\R$ s'identifie à une fonction de Busemann $b_\alpha$ (on note que la paire $(T_\alpha,b_\alpha)$ est isométrique à $(\T,b)$). Le morphisme $\iota$ induit une inclusion bilipschitzienne de $\Cone_\omega(G)$ dans $\T^W$. Il faut caractériser son image: on remarque d'abord que pour toute relation linéaire $\sum\lambda_\alpha p_\alpha=0$ vérifiée sur $\iota(G)$, l'image de $\Cone_\omega(G)$ dans le produit $\T^W$ est incluse dans le sous-ensemble de $\T^W$ d'équation $\sum\lambda_\alpha b_\alpha=0$. Une vérification élémentaire permet de vérifier qu'en fait l'image est précisément constituée des éléments $(x_\alpha)$ tels que les $(b_\alpha(x_\alpha))$ satisfont les mêmes relations linéaires que celles satisfaites par les $p_\alpha$ en restriction à $\iota(G)$. En conclusion, il existe un sous-espace vectoriel réel $V$ de $\R^W$ tel que le cône asymptotique $\Cone_\omega(G)$ s'identifie à l'ensemble des $(x_\alpha)_{\alpha\in W}\in \T^W$ tel que $\sum\lambda_\alpha b(x_\alpha)=0$ pour tout $(\lambda_\alpha)_{\alpha\in W}$ dans $V$.

Dans le cas général, si $M=\bigcap\Ker(\alpha)$ est le noyau de l'action (diagonalisable) de $N$ sur $E$, on voit $G$ comme le produit fibré de $G/M$ et $N=G/E$ au-dessus de $N/M$. Le cône asymptotique de $G/M$ est décrit comme ci-dessus comme sous-ensemble dans un produit d'arbres réels, et la projection de $G/M$ vers $N/M$ s'explicite comme une combinaison linéaire des fonctions de Busemann $b_\alpha$. Celui de $N$ est un groupe de Lie nilpotent simplement connexe avec une métrique de Carnot-Carathéodory, et la projection sur $N/M$ est la projection sur un quotient de son abélianisé. Une description explicite du cône asymptotique s'ensuit.

Remarquons que la description ci-dessus ne dépend pas de la dimension (non nulle) des $E_\alpha$.
\end{proof}

\begin{exe}\label{exv}
Soit $d\ge 1$ un entier fixé. Soit $G^d_0$ le groupe des matrices triangulaires supérieures de la forme
$$\begin{pmatrix} e^{t_1} & 0 & \cdots & 0 & x_1 \\  
					  & \ddots & \ddots& \vdots & \vdots \\
					 &  & \ddots & 0 & \vdots \\
					 & & & e^{t_d} & x_d \\ & & & & 1\end{pmatrix}, \;(t_1,\dots ,x_d)\in \R^{2d}.$$
Si $V$ un sous-espace vectoriel de $\R^d$, on s'intéresse au sous-groupe $G^d_V$ constitué de celles de ces matrices qui satisfont $\sum\lambda_i t_i=0$ pour tout $(\lambda_1,\dots,\lambda_d)\in V$, c'est-à-dire que $(t_1,\dots,t_d)$ est assujetti à appartenir à l'orthogonal $V^\bot$ de $V$. On suppose en outre (*) que $V^\bot$ n'est contenu dans aucun des hyperplans $(x_i=0)$, autrement dit que $V$ ne contient aucun des axes de base. Alors, en identifiant $V^\bot$ à l'ensemble des matrice diagonales de $G^d_V$ et en notant $U\simeq\R^d$ l'ensemble des matrices unipotentes de $G_0$, on a $G^d_V=U\rtimes V^\bot$, et $U$ est (grâce à (*)) le radical exponentiel de $G^d_V$. Alors, si $\T$ est l'arbre réel universel de valence $2^{\aleph_0}$ et $b$ une fonction de Busemann sur $\T$, tout cône asymptotique de $G^d_V$ est bilipschitzien au sous-ensemble
$$\T^d_V=\{(y_1,\dots,y_d)\in \T^d\mid (b(y_1),\dots,b(y_d))\in V^\bot\}.$$

La classe des groupes $G^d_V$ est une classe-test très naturelle pour la classification quasi-isométrique des groupes de Lie. De façon générale, je conjecture que deux groupes de Lie triangulables sont quasi-isométriques si et seulement s'ils sont isomorphes. On montre aisément que $G^d_V$ et $G^d_{V'}$ sont isomorphes si et seulement si $d=d'$ et $V$ et $V'$ sont dans la même orbite pour l'action du groupe symétrique $\mathfrak{S}_d$ sur la grassmannienne de $\R^d$; pour $d\ge 2$ il existe déjà une infinité de telles orbites.

En ce qui concerne le cône asymptotique, les choses sont différentes: en effet, pour tout $t>0$, il existe une auto-similitude de $\T$ envoyant $b$ sur $tb$. Une conséquence directe est que la classe modulo homéomorphisme bilipschitzien de $\T^d_V$ est constante sous les orbites du groupe des matrices diagonales à coefficients strictement positifs. On voit ainsi que pour $d\le 3$, on n'a qu'un petit nombre fini de possibilités (notons $v=\dim(V^\bot)$, qui est la dimension topologique du cône asymptotique). Notons $D(d)$ la droite de $\R^d$ engendrée par $(1,\dots,1)$.
\begin{itemize}
\item $v=1$ ($d$ quelconque) et $V^\bot$ est la droite engendrée par un vecteur à coefficients tous positifs (forcément tous strictement positifs par (*)); $\T^d_V$ est alors un arbre réel;
\item $v=1$ et $V^\bot$ est la droite engendrée par un vecteur à coefficients de signes non constants. Alors $\T^d_V$ est bilipschitzien à
$$\T^2_{D(2)}=\{(x,y)\in \T^2\mid b(x)+b(y)=0\},$$
qui est notamment le cône asymptotique de SOL, mais donc également de beaucoup d'autres groupes de dimension supérieure.
\item $d=v$, $\T^2_V$ est le produit de $d$ arbres réels.
\item $v=d-1$, $V$ est engendré par une droite à coefficients (strictement) positifs. Alors $\T^d_V$ est bilipschitzien à
$$\T^d_{D(d)}=\left\{(x_1,\dots,x_d)\in \T^d\;\Big|\; \sum_{i=1}^d b(x_i)=0\right\};$$
c'est là le cône asymptotique de $G^d_{D(d)}$ qui est un objet très étudié, son premier groupe d'homotopie non nul est précisément son $\pi_{d-1}$;
\item $v=d-1$, $V$ est engendré par une droite à coefficients de signe non constants. Si on note $D(k,\ell)$ la droite engendré par $(1,\dots,1,-1,\dots,-1)$ (avec $k$ nombres 1 et $\ell$ nombres $-1$), alors pour un $(k,\ell)$ tel que $k\ge\ell\ge 1$ et $k+\ell=d$, on voit que $\T^d_V$ est bilipschitzien à
$$\T^d_{D(k,\ell)}=\left\{(x_1,\dots,x_d)\in \T^d\;\Big|\; \sum_{i=1}^k b(x_i)=\sum_{i=k+1}^db(x_i)\right\};$$
c'est le cône asymptotique de $G^d_{D(k,\ell)}$, qui admet une métrique riemannienne invariante à courbure négative ou nulle \cite{AW}. En particulier, ce cône asymptotique est bilipschitzien à un espace métrique CAT(0) et en particulier est contractile.

\end{itemize}

(De façon générale, notons que par \cite{AW}, l'existence d'une métrique riemannienne invariante sur $G^d_V$ à courbure sectionnelle négative ou nulle équivaut à l'existence dans $V^\bot$ d'un vecteur à coefficients tous strictement positifs.)

On voit que ces cas particuliers recouvrent tous les cas où $v\in\{1,d-1,d\}$ et en particulier où $d\le 3$. En revanche, pour $(d,v)=(4,2)$, le groupe diagonal est de dimension 3 tandis que la 4-grassmanienne est de dimension 4, et les orbites ``génériques" sont de dimension 1. Il serait très intéressant de déterminer si on obtient ainsi une infinité de cônes asymptotiques non (bilipschitz) homéomorphes.
\end{exe}

Mentionnons enfin un problème concernant les cônes de dimension 1: il découle du théorème \ref{codimre} que si $G$ est un groupe de Lie triangulable (définition \ref{glt}), la dimension de son cône asymptotique est 1 (pour un ou tout ultrafiltre non principal) si et seulement s'il est un produit semi-direct $N\rtimes\R$, où $N$ est un groupe de Lie nilpotent, et l'action de $\R$ sur l'abélianisé de $N$ est sans point fixe non nul. 

Pour de tels groupes, il serait intéressant de classifier les cône asymptotiques. Si l'action de $\R$ sur $N$ est contractante, alors il est, comme déjà mentionné plus haut, isométrique à l'arbre réel $\mathbf{T}$, et ce pour tout choix de métrique géodésique et tout ultrafiltre non principal. Excluons ce cas maintenant. On peut vérifier qu'alors, le cône asymptotique possède un groupe fondamental non dénombrable et n'est pas semi-localement simplement connexe. On peut se demander si alors, en toute généralité, ce cône est homéomorphe (ou bilipschitzien) au cône de SOL (qui est lui unique à homéomorphisme bilipschizien près, par ce qui précède). L'argument qui précède montre que c'est le cas si $N$ a une décomposition en produit direct $\R$-invariante avec un facteur contracté et l'autre dilaté. Un exemple où ça ne s'applique pas est le produit semi-direct $H_3\rtimes\R$, où l'action de $\R$ sur $H_3$ est déterminée par la condition que le quotient par le centre de $H_3$ est isomorphe à SOL: dans ce cas-là déjà, on ne sait pas si le cône asymptotique pour un ultrafiltre donné est homéomorphe à $\Cone(\mathrm{SOL})$.

%
%

\chapter{Fonction de Dehn}\label{chafon}

Ce chapitre concerne les articles suivants (tous en collaboration avec R.~Tessera):
\begin{itemize}
\item \cite{CTcm} {\em Metabelian groups with quadratic Dehn function and Baumslag-Solitar groups}. Confluentes Math. 2010, 13 pages.
\item \cite{CTAB1} {\em Dehn function and asymptotic cones of Abels' group}. J. of Topology 2013, 27 pages.
\item \cite{CTge} {\em Geometric presentations of Lie groups and their Dehn functions}. Prépub\-lication, 115 pages, 2013.
\end{itemize}

L'article \cite{CTge} est l'aboutissement d'un projet établi en 2006: caractériser les groupes de Lie connexes ayant une fonction de Dehn exponentielle. Les articles \cite{CTcm,CTAB1} peuvent être vus comme des articles satellites de \cite{CTge}, en étudiant des cas particuliers; toutefois ils ne sont pas supplantés par \cite{CTge}, notamment parce que le fait d'étudier des cas particuliers permet de travailler sur des corps locaux de caractéristique non nulle, et beaucoup des corolaires qui sont obtenus dans \cite{CTge} sont spécifiques à la caractéristique non nulle.

Un point nouveau de tous ces articles est la manière de définir la fonction de Dehn, uniforme à tous les groupes localement compacts compactement présentés: un groupe localement compact est compactement présenté s'il possède une présentation de la forme $\langle S\mid\mathcal{R}\rangle$, où pour un certain $d$, l'ensemble $\mathcal{R}\subset F_S$ est l'ensemble des éléments de longueur $\le d$ relativement à la partie génératrice compacte $S$, dont l'image dans $G$ est triviale. Les éléments de $\mathcal{R}$ sont appelés {\em relateurs}, et les éléments du noyau $N$ du morphisme $F_S\to G$ sont appelés {\em relations}; par définition $N$ est engendré comme sous-groupe par l'ensemble $\mathcal{R}^{\textnormal{conj}}$ des éléments conjugués à un élément de $\mathcal{R}$. L'{\em aire}, notée $\textnormal{aire}(r)$ d'une relation $r\in N$ est sa longueur relativement à $\mathcal{R}^{\textnormal{conj}}$. La fonction de Dehn de $G$ (relative à $S$ et $d$) est définie par 
\[\delta(n)=\sup\{\mathrm{aire}(r)\mid r\in N,\;|r|_S\le n\}.\]
Un argument de compacité montre que $\delta(n)<\infty$ pour tout $n$. La croissance asymptotique de $\delta$ ne dépend pas du choix de $(S,d)$, et est en fait un invariant de quasi-isométrie de $G$.

Cette fonction a été largement étudiée pour $G$ discret. Pour $G$ groupe de Lie simplement connexe, une fonction dite de remplissage riemannien a été introduit par ailleurs, et il est connu que pour tout réseau cocompact $\Gamma$ de $G$, la fonction de Dehn de $\Gamma$ est asymptotiquement équivalente à la fonction de remplissage riemannienne de $G$. Le langage ci-dessus permet de décomposer ce résultat en deux résultats plus naturels: d'une part le fait que la fonction de remplissage riemannien de $G$ et sa fonction de Dehn soient asymptotiquement équivalentes (sans faire l'hypothèse artificielle et restrictive de l'existence d'un réseau), et d'autre part le fait que $G$ et tout réseau cocompact aient des fonctions de Dehn asymptotiquement équivalentes (car ils sont quasi-isométriques).

Ce langage permet également de prouver des estimations de fonction de Dehn pour des groupes algébriques sur des corps localement compacts (les corps localement compacts seront systématiquement supposés non discrets, par exemple, $\R$, $\Q_p$, ou $\mathbf{F}_p\lp t\rp$) de manière indépendante du corps, ou du moins par des méthodes générales, même si la conclusion dépend du corps.

Étant donné un corps localement compact $\K$ et $V$ un espace vectoriel de dimension finie sur $\K$, on dit qu'une action $\alpha:G\to\GL(V)$ d'un groupe sur $\K$ est {\em monomodulaire} si pour tout $g\in G$, toutes les valeurs propres de $\alpha(g)$ (dans une extension convenable) sont égales en valeur absolue, et que l'action est {\em distale} si toutes ces valeurs propres ont une valeur absolue égale à 1 pour tout $g$. Une action d'un groupe nilpotent sur un espace vectoriel se décompose toujours en une somme d'action monomodulaires \cite[\S 2]{CTge}, si bien que les hypothèses qui suivent ne sont pas spécialement restrictives.

On considère des espaces vectoriels $V_1,\dots,V_k$ de dimension finie non nulle sur des corps localement compacts $\K_i$. On considère un groupe abélien localement compact, compactement engendré $A$ (typiquement, $A=\Z^d$), avec des actions linéaires continues sur chacun des $V_i$. On suppose que pour tout $i$, l'action est monomodulaire et non distale. Soit
\[G=\left(\bigoplus V_i\right)\rtimes A.\]
On définit $A_i^+$ comme l'ensemble des éléments de $A$ qui agissent sur $V_i$ avec des valeurs propres de valeur absolue $>1$. Par hypothèse, $A_i$ est donc non vide pour tout $i$. Soit également $d$ la dimension sur $\R$ de $\Hom(A,\R)$.
\begin{thm}\label{Gij}Étant donné $G$ comme ci-dessus, on a
\begin{enumerate}
\item\label{ijna} S'il existe $i,j$ tel que $A_i^+\cap A_j^+=\emptyset$ et $\K_i$ et $\K_j$ sont non archimédiens, alors $G$ n'est pas compactement présenté;
\item\label{ija} S'il existe $i,j$ tel que $A_i^+\cap A_j^+=\emptyset$ mais (\ref{ijna}) ne s'applique pas, alors $G$ est compactement présenté de fonction de Dehn exponentielle.
\item\label{ijnv} Si pour tous $i,j$ on a $A_i^+\cap A_j^+\neq\emptyset$, alors $G$ a une fonction de Dehn linéaire si $d=1$ et quadratique si $d\ge 2$.
\end{enumerate}
\end{thm}

Le résultat (\ref{ijnv}) est extrait de \cite{CTcm}. Remarquons que pour tout $i$ il existe un morphisme continu non nul $\phi_i:A\to\R$ tel que $A_i^+=\phi_i^{-1}(\R_+^*)$; ainsi la condition $A_i^+\cap A_j^+=\emptyset$ revient à dire que dans l'espace vectoriel $\Hom(A,\R)$, le point 0 est dans le segment reliant $\phi_i$ à $\phi_j$. Ainsi la condition $A_i^+\cap A_j^+\neq\emptyset$ est générique si $d\ge 2$ mais pas si $d=1$.

Le résultat (\ref{ijna}) est conséquence de la version localement compacte du théorème de scindage de Bieri-Strebel \cite{BiS}; enfin (\ref{ija}) est prouvé dans \cite[\S 3.B]{CTge} pour la majoration et \cite[Chap.\ 8]{CTge} pour la minoration.

En appliquant ce résultat dans un cas particulier (un produit semi-direct convenable $\F_p\lp t\rp^3\rtimes\Z^2$) et en passant à un réseau cocompact, on en déduit

\begin{cor}[\cite{CTcm}]
Le groupe de Baumslag $\F_p[t,t^{-1},(t-1)^{-1}]\rtimes_{(t,t-1)}\Z^2$ a une fonction de Dehn quadratique.
\end{cor}

Également à l'aide des nombres $p$-adiques, on obtient le corollaire suivant:

\begin{cor}[\cite{CTcm}]
Pour tout $n\ge 2$, le groupe de Baumslag-Solitar $\mathrm{BS}(1,n)$ est isomorphe à un sous-groupe d'un groupe de présentation finie à fonction de Dehn quadratique.
\end{cor}

Cette question avait été posée dans \cite{BORS} et des résultats partiels (avec fonction de Dehn polynomiale, puis cubique, au lieu de quadratique) avaient été obtenus dans \cite{BORS} et \cite{AO}. Ici, le groupe de présentation finie obtenu est un produit semi-direct convenable $\Z[1/n]^3\rtimes\Z^2$.


Des résultats dans l'esprit du \ref{Gij} s'obtiennent pour des groupes non métabéliens. Plus précisément, considérons un produit semi-direct $U\rtimes A$, où $A$ est toujours un groupe localement compact abélien, et $U$ possède une filtration $\{1\}=U_0\subset U_{[1]}\subset\dots \subset U_{[k]}=U$, où chaque $U_i$ est un sous-groupe fermé $A$-invariant, et chaque $U_i=U_{[i]}/U_{[i-1]}$ est un $\K_i$-espace vectoriel sur $\K_i$ muni d'une action continue monomodulaire non distale de $A$. 

%

\begin{thm}[{\cite[\S 4.D]{CTge}}]\label{stam}
Supposons de plus que tous les $\K_i$ sont de caractéristique 0. Supposons que pour tous $i,j$ on a $A_i^+\cap A_j^+\neq\emptyset$. Alors $G$ a une fonction de Dehn linéaire si $d=1$ et quadratique si $d=2$.
\end{thm}

L'hypothèse de caractéristique 0 est certainement superflue mais quelques préli\-mi\-naires algébriques seraient nécessaires pour englober la caractéristique non nulle, l'argument géométrique restant le même que dans \cite[\S 4.D]{CTge} (mais pas que dans \cite{CTcm}, la commutativité du $U$ dans le théorème \ref{Gij} escamotant une partie de l'argument).

Commentons l'hypothèse ``non distale": dans le théorème \ref{Gij} cette hypothèse est naturelle: en effet si elle est omise et que l'action de $A$ sur $V_i$ est distale, alors si $\K_i$ est non archimédien, alors $G$ n'est pas compactement engendré; si $\K_i$ est archimédien, il y a une discussion un peu plus compliquée mais ce qui se passe est en tout cas bien compris. En revanche, dans le cadre du théorème \ref{stam}, l'hypothèse de non distalité est restrictive, et de même, l'hypothèse $A_i^+\cap A_j^+\neq\emptyset$ pour tous $i,j$ l'est également, car contrairement au théorème \ref{Gij}, la réciproque est fausse. Le premier exemple étudié en ce sens est le groupe d'Abels, considéré dans \cite{CTAB1}, dont nous présentons quelques résultats ci-dessous.

Si $R$ est un anneau commutatif et $d\ge 2$, on définit le groupe $A_d(R)$ comme l'ensemble des matrices inversibles triangulaires supérieures à coefficients dans $R$, dont les coefficients $(1,1)$ et $(d,d)$ sont égaux à 1. Par exemple,
\[A_4(R)=\left\{\begin{pmatrix}1 & u_{12} & u_{13} & u_{14} \\ 0 & s_{22} &  u_{23 }& u_{24}\\ 0 & 0 & s_{33} & u_{34} \\ 0 & 0 & 0 & 1\end{pmatrix}: u_{ij}\in R,\;s_{ii}\in R^\times\right\}\]

Le groupe $A_3$ a en fait été introduit par P.~Hall, et Abels a plus tard considéré $A_d$ pour $d\ge 4$ \cite{Ab1}, établissant le fait remarquable et alors inattendu est que $A_d(\Z[1/p])$ est de présentation finie si (et seulement si) $d\ge 4$ (il est de type fini si et seulement si $d\ge 3$). 

\begin{thm}\cite{CTAB1}
Pour tout corps localement compact non discret $\K$ et $d\ge 4$, le groupe $A_d(\K)$ a une fonction de Dehn quadratique.
\end{thm}

Le théorème \ref{stam} ne s'applique pas, car le centre de $A_d(\K)$, constitué des matrices de la forme $e_{1d}(\lambda)$ pour $\lambda\in\K$, est inclus dans le sous-groupe dérivé.  

Dans \cite{CTAB1} on en tire diverses applications parfois surprenantes, notamment

\begin{cor}
Il existe un groupe polycyclique dont tout cône asymptotique a un groupe fondamental abélien non trivial.
\end{cor}

En effet parmi les exemples précé\-demment connus parmi les groupes polycycliques, ou bien le groupe fondamental de tout cône asymptotique est trivial, ou bien le groupe fondamental de tout cône asymptotique contient un groupe libre de rang non dénombrable.

Passons maintenant aux résultats plus difficiles de \cite{CTge}. On va se restreindre tout de suite à des groupes de la forme $G=U\rtimes A$, où $U$ est un produit de groupes unipotents sur divers corps localement compacts non discrets de caractéristique zéro, et $A$ est abélien. Cette hypothèse permet de se concentrer sur le fond du problème; dans le cas d'un groupe algébrique $p$-adique, on peut s'y ramener en passant à un sous-groupe cocompact; dans le cas d'un groupe de Lie réel connexe, on peut s'y ramener, mais de manière plus complexe, et avec un coût polynomial sur la fonction de Dehn \cite[Chap.\ 3]{CTge}. Le groupe $U$ possède un dévissage $A$-équivariant en des sous-quotients $U_i$ abéliens, de manière que chaque $U_i$ est un espace vectoriel non nul sur un corps localement compact non discret $\K_i$ de caractéristique zéro, de manière que l'action de $A$ sur $U_i$ soit monomodulaire (mais pas forcément non distale). Pour $g\in A$, on note $\alpha_i(g)$ le logarithme du module commun des valeurs propres de $A$ sur $U_i$. Ainsi $\alpha_i$ est un morphisme continu de $A$ dans $\R$ (nul si et seulement si l'action sur $U_i$ est distale). Les $\alpha_i$ sont appelés {\em poids} de $G$ (par abus, puisque cela dépend a priori de la décomposition $G=U\rtimes A$, quoique de façon non essentielle); ils vivent dans l'espace vectoriel réel de dimension finie $\Hom(A,\R)$ (qui lui-même est inclus dans $\Hom(G,\R)$, ce dernier ne dépendant pas des choix). Les $\alpha_i$ tels que $\K_i$ est non-archimédien sont appelés {\em poids non-archimédiens}. Les poids de $G/[U,U]$ sont appelés {\em poids principaux}. On supposera que 0 n'est pas un poids principal; sous cette hypothèse on dira que $G$ est {\em résoluble standard}. À nouveau, ce n'est pas une restriction essentielle et on peut s'y ramener. Sous ces hypothèses, l'ensemble des poids et des poids principaux, vu comme sous-ensemble de $\Hom(G,\R)$, ne dépend pas de la décomposition en produit semi-direct. La donnée des poids gouverne de manière essentielle la géométrie du groupe $G$. Notons que le semi-groupe additif engendré par les poids principaux contient tous les poids.

Par exemple, la condition que 0 n'est pas dans l'enveloppe convexe des poids (ou de manière équivalente, des poids principaux), est très naturelle. Elle a été considérée dans le contexte des groupes de Lie par Azencott-Wilson et Varopoulos; dans le contexte des groupes $p$-adiques par Abels. Si cette condition est vérifiée pour un groupe résoluble standard, on dit que $G$ est {\em modéré} (en anglais: {\em tame}). 

\begin{exe}
Pour tout $d\ge 0$, dans $\SL_d(\R)$, le groupe des matrices triangulaires supérieures est résoluble standard modéré. Par exemple, le groupe affine de la droite est modéré. En revanche, le groupe SOL ne l'est pas; les groupes d'Abels $A_d(\R)$ non plus pour $d\ge 3$.
\end{exe}


L'hypothèse du théorème \ref{stam} s'exprime simplement en termes de poids. Pour cela introduisons la terminologie: on dit que $G$ est fortement 2-modéré si 0 n'est dans l'enveloppe convexe d'aucune paire de poids, et 2-modéré si 0 n'est dans l'enveloppe convexe d'aucune paire de poids principaux. Remarquons que 
\[\textnormal{modéré}\Rightarrow\textnormal{fortement 2-modéré}\Rightarrow\textnormal{2-modéré}\]
L'hypothèse du théorème \ref{stam} est simplement que $G$ est fortement 2-modéré. Comme on a indiqué précé\-demment, cette hypothèse est restrictive, ainsi pour $d\ge 4$ le groupe d'Abels $A_d(\R)$ est 2-modéré mais pas fortement 2-modéré: en effet, par exemple on peut dessiner les poids du groupe d'Abels $A_4(\R)$, qui vivent dans un plan, comme ceci:

$$\begin{array}[c]{ccccc} 
 && \mathbf{23} &&\\
 13 &&&& 24\\
 && \underline{14} &&\\
 \mathbf{12} &&&& \mathbf{34}
\end{array}$$
Ici les poids principaux sont indiqués en gras, les autres sont non principaux et le 14 souligné indique qu'il s'agit du poids nul. 

Il découle des résultats d'Abels que si $G/G^\circ$ n'est pas 2-modéré, alors $G$ n'est pas compactement présenté. La preuve se ramène essentiellement à utiliser une version localement compacte du théorème de scindage de Bieri-Strebel \cite{BiS}.

Dans \cite{CTge}, on démontre

\begin{thm}
Si $G$ (résoluble standard) n'est pas 2-modéré, sa fonction de Dehn est au moins exponentielle.
\end{thm}

Ici on n'exclut pas la possibilité que $G$ soit non compactement présenté, auquel cas on convient de définir sa fonction de Dehn comme la fonction identiquement égale à $+\infty$.

Il existe une autre condition importante, provenant des extensions centrales. On considère l'algèbre de Lie $\mathfrak{u}$ de $U$ (qui est une algèbre de Lie sur le produit de corps localement compacts qu'il faut). Elle admet une graduation naturelle dans $\Hom(G,\R)$, dictée par les poids. L'homologie de $\mathfrak{u}$ est également graduée (toujours dans $\Hom(G,\R)$). Sa composante en degré zéro est intimement reliée aux extensions centrales de $G$ comme groupe topologique. 

Abels observe que si $H_2(\mathfrak{u}/\mathfrak{u}^\circ)_0\neq 0$, alors $G$ n'est pas compactement présenté. La raison est que cette condition entraîne, à peu de choses près, l'existence d'une extension centrale compactement engendrée, dans laquelle le noyau central n'est pas compactement engendré. Un argument similaire \cite[Chap.\ 7]{CTge} démontre que si $H_2(\mathfrak{u})_0\neq 0$, alors $G$ a une fonction de Dehn au moins exponentielle.

Les résultats négatifs précé\-demment décrits ne sont pas très difficiles; le travail d'Abels pour la présentation compacte puis de Tessera et moi pour la fonction de Dehn, est de montrer que lorsque ces obstructions (non 2-modération ou existence de 2-homologie) ne s'appliquent pas, alors les choses se passent bien. Le principal théorème d'Abels est ainsi celui-ci: si $G/G_0$ est 2-modéré et $H_2(\mathfrak{u}/\mathfrak{u}^\circ)_0=0$, alors $G$ est compactement présenté.

Nos résultats précisent ces résultats au niveau de la fonction de Dehn.

\begin{thm}Soit $G$ un groupe standard résoluble comme ci-dessus.
\begin{enumerate}
\item\label{expmaj} Si $G/G_0$ est 2-modéré et $H_2(\mathfrak{u}/\mathfrak{u}^\circ)_0=0$ alors $G$ a une fonction de Dehn au plus exponentielle;
\item\label{cubi} Si $G$ est 2-modéré et $H_2(\mathfrak{u})_0=0$ alors $G$ a une fonction de Dehn bornée cubiquement.
\end{enumerate}
\end{thm}

Ici le principal résultat est (\ref{cubi}). Ici (\ref{expmaj}) combine 2 ingrédients: 
\begin{itemize}\item le fait d'appliquer (\ref{cubi}) à $G/G^\circ$ (dont on n'utilise que la majoration exponentielle, mais qui n'est pas immédiate ni ne découle des travaux d'Abels),
\item et le fait que $U^\circ\rtimes A$ a une fonction de Dehn au plus exponentielle. Ceci découle du fait (plus général) que la fonction de Dehn d'un groupe de Lie connexe arbitraire est au plus exponentielle, résultat affirmé par Gromov avec une esquisse de preuve.
\end{itemize}

Il reste à discuter (\ref{cubi}). Sa preuve consiste dans un premier temps à redémontrer le théorème d'Abels (la présentation compacte), mais en utilisant une ``présentation" qui soit ``utile" pour la suite. Pour cela, on suit les pas suivants, que l'on comparera ensuite à la démarche d'Abels (voir \cite[\S 4.H]{CTge}).


\begin{enumerate}[(a)]
\item\label{amal} On considère les sous-groupes modérés maximaux $G_i=U_i\rtimes A$ de $G$. Ici, les $U_i$ sont des sous-groupes de $U$ faciles à déterminer; ils sont en nombre fini. On considère l'amalgame des $\hat{G}$ des $G_i$ le long de leurs intersections; ainsi $\hat{G}=\hat{U}\rtimes A$, où $\hat{U}$ est l'amalgame des $U_i$ le long de leurs intersections. Ce groupe a un morphisme canonique surjectif vers $G$. 
\item\label{2modc} Sous l'hypothèse que $G$ est 2-modéré, on montre que l'extension $\hat{U}\to U$ est centrale.
\item En utilisant de plus l'hypothèse que $H_2(\mathfrak{u})_0=0$, on décrit des générateurs pour le noyau de $\hat{G}\to G$, qu'on appelle ``relations de soudage".
\end{enumerate}

Comparons la démarche à celle d'Abels: (\ref{amal}) est exactement son approche; à ceci près qu'il travaille dans un cadre $p$-adique. Abels ne démontre pas (\ref{2modc}) et demande explicitement si c'est vrai. Il démontre toute fois le résultat analogue en termes d'algèbres de Lie, et au niveau des groupes démontre que $\hat{U}$ est nilpotent ($(s+1)$-nilpotent si $U$ est $s$-nilpotent). Nous utilisons ces deux résultats partiels dans notre preuve que $\hat{U}\to U$ est une extension centrale.

En tant qu'amalgame de groupes compactement présentés au dessus de sous-groupes compactement engendrés, le groupe $\hat{G}$ a une présentation bornée au dessus d'un système de générateur compact. Ainsi si $\hat{G}\to G$ est un isomorphisme, $G$ est compactement présenté, et notre argument est dans ce cas le même que celui d'Abels. Toutefois ce n'est pas toujours le cas. La démarche restante diffère drastiquement: Abels (qui travaille avec $G$ totalement discontinu) considère un ``amalgame" de $\hat{G}$ et d'un sous-groupe compact ouvert au dessus de leurs intersections et montre, en utilisant que $H_2(\mathfrak{u})_0=0$, que l'amalgame obtenu est $G$.

Cette dernière approche n'est pas assez algébrique pour qu'on puisse l'exploiter pour estimer la fonction de Dehn. On explicite, à la place, des générateurs pour le noyau de $\hat{U}\to U$. Ces générateurs ont une forme simple dans le contexte de l'algèbre de Lie: l'algèbre de Lie $\hat{\mathfrak{u}}$ de $\hat{U}$ (qui a bien un sens car on démontre que $\hat{U}$ est un groupe nilpotent uniquement divisible) est une algèbre de Lie sur les rationnels; le noyau (central) de $\hat{\mathfrak{u}}\to\mathfrak{u}$ est engendré comme groupe par les élements la forme $[\lambda x,y]-[x,\lambda y]$ où $x,y$ parcourent les éléments de poids non nuls opposés et $\lambda$ parcourt les scalaires. La formule de Baker-Campbell-Hausdorff inverse permet de traduire ces termes à l'aide des éléments des $G_i$ et de la loi de groupe uniquement. Les formules obtenues sont assez indigestes mais sont des formules ``closes", ne dépendant que de la classe de nilpotence; ainsi les éléments du noyau sont de longueur bornée par rapport au système générateur et le groupe $G$ est compactement présenté.

Maintenant pour expliquer comment on en déduit une estimation de la fonction de Dehn, on fait la démarche qui précède en ``changeant d'algèbre". Pour simplifier l'exposition de l'esquisse qui suit (voir \cite[\S 4.K]{CTge} pour un argument rigoureux), on va supposer que tous les corps $\K_i$ sont égaux à un même $\K$. On peut écrire $U=\mathbb{U}(\K)$, où $\mathbb{U}$ est un groupe algébrique unipotent défini sur $\K$. Le point est de faire l'approche qui précède simultanément pour toutes les $\K$-algèbres commutatives $R$: on définit ainsi l'amalgame $\widehat{\mathbb{U}(R)}$ des $\mathbb{U}_i(R)$. On pose $G(R)=\mathbb{U}(R)\rtimes A$ et $\widehat{G(R)}=\widehat{\mathbb{U}(R)}\rtimes A$. Alors tout ce qui précède reste vrai: si $G$ est 2-modéré alors le noyau $Z(R)$ de $\widehat{\mathbb{U}(R)}\to\mathbb{U}(R)$ est central, et si de plus $H_2(\mathfrak{u})_0=0$, alors ce noyau est engendré par les ``relations de soudure" (paramétrées par un produit convenable des $\mathbb{U}_i(R)$, où les $\mathbb{U}_i$ et la paramétrisation ne dépendent pas de $R$).

Ce résultat de présentation, uniforme sur toutes les algèbres, s'applique lorsque $R$ est l'algèbre $\K^{\textnormal{pol}}$ des suites d'éléments de $\K$ croissant au plus polynomialement: si on a une suite de lacets combinatoires dans $G$ de longueur croissant linéairement, on doit commencer par l'interpréter comme un seul lacet combinatoire dans $G(\K^{\textnormal{pol}})$. Ceci n'est pas possible pour une suite de lacets combinatoires arbitraire, mais est faisable pour une suite de lacets combinatoires $(u_n)$ de la forme $u_n=\prod_{j=1}^M \gamma_js_j\gamma_j^{-1}$, où $M$ ne dépend pas de $n$, les $\gamma_i$ sont des mots dans $A$ (de taille au plus linéaire) et les $s_i$ sont des générateurs dans $\bigcup U_i$. L'astuce de Gromov permet en effet de se ramener à des mots de cette forme. 

Une fois $(u_n)$ interprété comme un lacet combinatoire dans $G(\K^{\textnormal{pol}})$, celui-ci est conséquence des relateurs: cela permet de le décomposer comme un produit de conjugués de relateurs à coefficients dans $\K^{\textnormal{pol}}$. Cela décrit chaque $u_n$ comme un produit borné de conjugués de relations spéciales, à savoir d'une part les relations à l'intérieur des sous-groupes modérés (qui ont une aire au plus quadratique), d'autre part les relations de soudage. Un calcul explicite, également basé sur certains préliminaires algébriques, démontre que l'aire des relations de soudage est au plus cubique.

Pour conclure, il suffit de choisir pour $u_n$ une relation, parmi celles de forme spécifiée par l'astuce de Gromov, d'aire maximale parmi les relations de taille au plus $n$. On obtient ainsi que l'aire de $u_n$, c'est-à-dire la fonction de Dehn $\delta(n)$, est au plus cubique!

\bigskip

Terminons par mentionner quelques problèmes. L'un est que nous ne savons pas si la majoration asymptotique cubique est parfois optimale. Cela pourrait dépendre d'invariants fins de l'algèbre de Lie graduée $\mathfrak{u}$.

Un autre problème serait d'appliquer des méthodes similaires en vue de l'estimation de la fonction de Dehn de groupes de Lie nilpotents: rappelons que cette dernière n'est estimée précisément que dans des cas très particuliers. Par exemple, si $U$ est l'unipotent maximal du groupe des isométries du plan hyperbolique octonionique (ce dernier étant localement isomorphe au groupe de Lie exceptionnel $F_{4(-20)}$), $U$ est 2-nilpotent ($Z(U)=[U,U]$ étant de dimension 7 et $U/[U,U]$ de dimension 8), la fonction de Dehn n'est pas connue: elle est au plus cubique comme pour tout groupe 2-nilpotent, et strictement plus grande que quadratique par un résultat de Wenger~\cite{Wen}.


\chapter[Groupes hyperboliques moyennables]{Groupes localement compacts hyperboliques moyennables}

Ce chapitre est une synthèse des articles suivants:
\begin{itemize}
\item \cite{CAM} (avec R.~Tessera) {\em Contracting automorphisms and $L^p$-cohomology in degree one}. Ark. Mat., 30 pages (2011),
\item \cite{CCMT} (avec P-E.~Caprace, N.~Monod, R.~Tessera) {\em Amenable hyperbolic groups}. 40 pages, soumis.
\item \cite{CQF} {\em On the commability and quasi-isometry classification of focal groups}. 24 pages, soumis.
\end{itemize}

Le point de départ de ces travaux est une esquisse par Pierre Pansu en réponse à une question que je lui ai posé (nov.\ 2005), d'une caractérisation, parmi les groupes de Lie connexe, de ceux qui sont Gromov-hyperboliques. La réponse étant, en terme vagues: ``essentiellement" les seuls sont ceux qui admettent une métrique riemannienne invariante à courbure sectionnelle strictement négative. Ces derniers ont été caractérisés par Heintze: en dimension $\le 2$, ce sont ceux qui admettent une décomposition en produit semi-direct $U\rtimes\R$, l'action de $\R$ sur le groupe nilpotent simplement connexe $U$ étant contractante. Appelons ces derniers groupes de Heintze (en y adjoignant $\R$ et le groupe trivial). Le lemme \ref{lemtr} (ou une variante idoine) permet de ramener au problème précis suivant: montrer que tout groupe triangulable non de Heintze n'est pas Gromov-hyperbolique.

Mon approche était la suivante: on sait que pour un groupe localement compact Gromov-hyperbolique, les cônes asymptotiques sont des arbres réels (et en particulier sont simplement connexes, et de dimension topologique 1) et la fonction de Dehn croît linéairement. Or si $G$ est un groupe de Lie triangulable et $E$ son radical exponentiel, alors 
\begin{itemize}
\item soit $\dim(G/E)\ge 2$ et alors les cônes aysmptotiques de dimension $\ge 2$ par le résultat principal de \cite{CJT};
\item soit $\dim(G/E)=1$, et alors soit $G$ est de Heintze, soit ses cônes asymptotiques sont non simplement connexes et sa fonction de Dehn croît exponentiellement. Avec Romain Tessera, nous avons montré ces deux résultats, mais seul le second a été rédigé pour l'instant \cite[Section 8]{CTge}, et ce postérieurement à l'approche par la cohomologie $L^p$ décrite ci-après.
\end{itemize}

L'approche de Pierre Pansu, complétée par Tessera, est différente, et donne un point de vue complémentaire.
\begin{itemize}
\item Le premier point était de vérifier que pour un groupe localement compact Gromov-hyperbolique $G$, on a $\overline{H^{1,p}}(G)\neq 0$ si $p$ est assez grand. Ici $\overline{H^{1,p}}(G)$ est la cohomologie $L^p$ réduite  en degré 1, que l'on convient ici de définir comme la cohomologie réduite de la représentation régulière droite de $G$ dans $L^p(G)$, $G$ étant muni d'une mesure de Haar à gauche. Ce résultat était connu pour $G$ discret et le cas général, vérifié par Tessera \cite[Theorem 9.2]{T}, est basé sur la même méthode.
\item Le second point est de montrer que si $G$ est un groupe de Lie triangulable et n'est pas de Heintze, alors sa cohomologie réduite en degré 1 $\overline{H^{1,p}}(G)$ s'annule pour tout $1\le p<\infty$. On distingue deux cas:
\begin{itemize}
\item $G$ n'est pas unimodulaire. Alors cette annulation est un résultat de Pansu (datant de 1995, mais publié en 2007 \cite{Pan07} avec en vue cette application à la non-hyperbolicité.)
\item $G$ est unimodulaire. Alors cette annulation est le résultat principal de Tessera \cite{T}.
\end{itemize}
\end{itemize}

Le travail de \cite{CAM} est le suivant: d'une part, il donne un énoncé précis caractérisant les groupes de Lie connexes Gromov-hyperboliques (pas forcément triangulables), l'énoncé (provisoire) de \cite{T} étant le fait d'être quasi-isométrique à un groupe de Heintze. 

\begin{thm}[Cornulier-Tessera \cite{CAM}]
Un groupe de Lie connexe $G$ est Gromov-hyperbolique non élémentaire si et seulement si l'une (et une seule!) des deux conditions suivantes est vérifiée:
\begin{itemize}
\item $G$ est isomorphe à un produit semi-direct $N\rtimes (K\times\R)$ où $N$ est un groupe de Lie simplement connexe nilpotent, $K$ est un groupe de Lie compact connexe, et l'action des éléments positifs de $\R$ sur $N$ est contractante;
\item le quotient de $G$ par son sous-groupe compact distingué maximal est un groupe de Lie connexe simple adjoint de rang 1.
\end{itemize}
\end{thm}

L'autre part du travail dans \cite{CAM} est de reprouver les résultats d'annulation et de non-annulation de cohomologie $L^p$ en degré 1 de Pansu \cite{Pan07} en utilisant la représentation régulière droite, permettant de les étendre en dehors du contexte des groupes de Lie. La définition utilisée par Pansu invoque l'espace des fonctions $L^p$ dont le gradient (au sens des distributions) est $L^p$. La représentation régulière droite permet de définir $\overline{H^{1,p}}(G)$ comme ceci: ($G$ est muni d'une mesure de Haar à gauche, $1\le p<\infty$ est fixé et $\rho$ est la représentation régulière droite: $\rho(g)f(h)=f(hg)$): on définit $D^p(G)$ comme l'ensemble des fonctions $f\in L^1_\textnormal{loc}(G)$ telles que $g\mapsto f-\rho(g)f$ est une fonction continue de $G$ dans $L^p(G)$. C'est un espace vectoriel topologique, où $f_i$ tend vers 0 si $f_i-\rho(g)f_i$ tend vers 0 dans $L^p$, uniformément en $g$ sur les compacts; le quotient de $D^p(G)$ par la fermeture de $L^p(G)$ est par définition la cohomologie réduite de la représentation régulière droite; on définit ici $\overline{H^{1,p}}(G)$ de cette manière, l'équivalence avec la définition utilise par Pansu étant établie dans \cite[\S 5]{T} (bien que n'étant pas utilisée dans les applications ci-dessus). 

Le résultat d'annulation dans le cas non-unimodulaire non-Heintze est ainsi étendu, avec une preuve unifiant les cas réels et $p$-adiques, dans \cite[Theorem 2.16]{CAM}. Insistons ici plutôt sur le cas Heintze et son extension dans un cadre plus général, car ils ont des applications plus percutantes à la classification quasi-isométrique.

\begin{defn}
Un automorphisme $\alpha$ d'un groupe localement compact $N$ est une {\em compaction} s'il existe une partie compacte $\Omega\subset G$ qui est $\alpha$-aspirante, au sens où pour toute partie  compacte $K\subset G$, il existe un entier $n\ge 0$ tel que $\alpha(K\cup\Omega)\subset\Omega$. Si, de plus, tout voisinage de 1 dans $G$ est $\alpha$-aspirant, on dit que $\alpha$ est un automorphisme {\em contractant}.
\end{defn}

\begin{rem}Les automorphismes contractants sont très étudiés et bien compris. Bien que beaucoup moins étudiés, les automorphismes compactants étaient précé\-demment connus sous le nom d'automorphisme {\em contractant modulo un sous-groupe compact}, et sont appelés {\em contractants} dans (et uniquement dans) \cite{CAM}. La terminologie ``compactants" est due à \cite{CCMT}. 
\end{rem}

Dans \cite[Appendix A]{CAM}, on étend au cas compactant un résultat de décomposition directe (connexe / totalement discontinu) des contractions, à l'aide duquel on démontre le théorème suivant (dû à Pansu pour un groupe de Lie connexe).

\begin{thm}[{\cite[Theorem 7]{CAM}}]\label{hp0}
Soit $G=N\rtimes\Z$ (ou $N\rtimes\R$) un produit semi-direct, l'action des éléments positifs de $\Z$ ou $\R$ sur $N$ étant compactante. Alors il existe un nombre réel (explicite) $p_0(G)$ tel que pour tout $p\in [1,\infty\mathclose[$, on a $H^{(1,p)}(G)\neq 0$ si et seulement si $p>p_0(G)$.
\end{thm}

Le nombre $p_0(G)$ s'explicite comme ceci: d'abord si $G_0$ (la composante connexe de l'unité) est compact on prend $p_0=0$ (ce cas est facile, laissons-le donc de côté). Sinon, soit $t\in G$ un élément strictement négatif du sous-groupe $\Z$ ou $\R$ (selon le cas), soit $\lambda>1$ la plus petite valeur propre de l'action de $\Z$ sur l'algèbre de Lie de $G_0/W$ ($W$ étant son plus gros sous-groupe compact distingué), et soit $\delta\ge\lambda$ la multiplication du volume de $t$ dans $G_0$. Alors $p_0(G)=\log(\delta)/\log(\lambda)\ge 1$ ne dépend pas du choix de $t$ et satisfait le théorème précédent.

Dans le cas où $N$ est un groupe de Lie connexe, $\delta$ est un produit de valeur propres et c'est de cette manière que le résultat est énoncé par Pansu. Ici, $\delta=\delta_{\textnormal{c}}\delta_{\textnormal{td}}$ où $\delta_{\textnormal{c}}$ est la contribution de la partie connexe (un produit de valeurs propres), et $\delta_{\textnormal{td}}$ est un entier $\ge 1$, contribution de la partie totalement discontinue.

Une motivation pour un énoncé sous l'hypothèse du théorème \ref{hp0} est le théorème suivant, qui n'était alors, dans toute sa généralité, qu'une conjecture (à Tessera et moi-même).

\begin{thm}[\cite{CCMT}]\label{mainccmt}
Un groupe localement compact moyennable est Gromov-hyper\-bolique non élémentaire si et seulement s'il est isomorphe à un TAC (tore d'application d'une compaction), au sens où il admet une décomposition en produit semi-direct $G=N\rtimes\Z$ ou $N\rtimes\R$, l'action des éléments positifs de $\Z$ ou $\R$ sur $N$ étant compactante. 
\end{thm}

L'implication que nous ne savions pas prouver en écrivant \cite{CAM} est l'existence d'une telle décomposition, pour un groupe Gromov-hyperbolique moyennable non élémentaire. Il s'agit dans un premier temps de mettre la main sur un morphisme vers $\R$ ou $\Z$. Il est tentant de considérer la fonction modulaire, mais nous sommes arrivés au bout en considérant plutôt le quasi-caractère de Busemann, relativement à un point fixe sur le bord du groupe (ce point fixe existe, car sinon un argument simple de ping-pong dû à Gromov \cite{Gro87} permettrait de construire un sous-groupe discret libre non abélien dans $G$, niant la moyennabilité).  

Pour l'autre implication, une preuve directe se trouve dans \cite{CCMT}: elle a l'avantage d'être purement métrique et de n'utiliser aucun théorème de structure sur les groupes localement compacts et donc de s'étendre à un cadre plus général (bien qu'à cette heure on n'ait pas encore sérieusement exploré cette direction). Une autre preuve, plus élaborée, se trouve également dans \cite{CCMT}, mais fournit une conclusion plus forte.

\begin{thm}\label{tacmi}
Soit $G$ un TAC (tore d'application de compaction, au sens du théo\-rème \ref{mainccmt}) admet une action continue isométrique, propre et cocompacte sur un espace CAT(-1) propre $X$, dit espace ``millefeuille".
\end{thm}

Pour expliquer ce qu'est un espace millefeuille, commençons par deux cas particuliers dégénérés:
\begin{itemize}
\item si $N$ est connexe (ou plus généralement $\pi_0(N)$ est compact), alors $X$ est une variété riemannienne (complète simplement connexe) à courbure sectionnelle strictement négative.
\item si $N_0$ est compact, alors $X$ est un arbre ($G$ se voit en effet alors comme extension HNN ascendante d'un groupe compact, dont on prend l'arbre de Bass-Serre).
\end{itemize}
Un millefeuille, en général, est un mélange de ces deux cas: on considère une variété riemannienne $X$ à courbure sectionnelle strictement négative (simplement connexe et complète), et une fonction de Busemann $b$. On considère d'autre part l'arbre $T_k$ régulier de valence $k+1$, muni d'une fonction de Busemann $b'$. Alors $X_b[k]$ est par définition défini comme
$$X_b[k]=\{(x,y)\in X\times T_k\mid b(x)=b'(y)\};$$
il est muni d'une métrique de longueur de sorte que pour toute géodésique $D$ de $T_k$ en restriction de qui $b'$ est bijective, la projection (bijective) de la ``feuille" $X_b[k]\cap (X\times D)$ sur $X$ soit une isométrie. Ainsi, $X_b[k]$ est un recollement itéré de ces feuilles en des parties convexes fermées isométriques; si $X$ est CAT($-\kappa$) avec $-\kappa<0$ alors $X_b[k]$ aussi. Remarquons que topologiquement, $b$ s'identifie à une projection $\R^d\to\R$ et il en découle que $X_b[k]$ est homéomorphe à $\R^{d-1}\times T_k$ (ici $T_k$ est identifié à son 1-squelette).

En outre, si $X$ admet un groupe cocompact d'isométries préservant la classe de $b$ modulo addition de fonctions constantes, alors $X_b[k]$ aussi. On obtient ainsi une classe d'espaces métriques propres à courbure strictement négative ayant un groupe d'isométries cocompact. Les millefeuille purs sont ceux qui ne sont pas dégénérés (i.e.\ réduit à une variété ou à un arbre), ou, de façon équivalente, satisfaisant $k\ge 2$ et $\dim(X)\ge 2$.

Les espaces millefeuilles ont été introduits dans \cite{CCMT}; leur définition ressemble fortement à l'espace \og arbolique\fg  ({\em treebolic})  $\{(x,y)\in X\times T_k\mid b(x)+b'(y)=0\}$ considéré par exemple dans \cite{BSSW} qui ont la même topologie, mais le changement de signe modifie drastiquement la géométrie, puisqu'on recolle des complémentaires d'horoboules (non convexes) au lieu d'horoboules et ces espaces ``arboliques" ne sont essentiellement jamais Gromov-hyperboliques.

Le papier \cite{CQF} s'intéresse à la classification des TAC (appelés ``groupes focaux" dans \cite{CQF}) à quasi-isométrie près.


Si on a un TAC de la forme $G=N\rtimes\Z$ ou $G=N\rtimes\R$, on dit qu'il est {\em de type connexe} (resp.\ {\em totalement discontinu}) si $N/N_0$ est compact (resp.\ si $N_0$ est compact.) Sinon, on dit que $G$ est {\em de type mixte} (notons qu'un TAC de la forme $N\rtimes\R$ est forcément de type connexe). 

Tous les TAC de type totalement discontinu admettent une action géométrique (c'est-à-dire continue, isométrique, propre et cocompacte) sur un arbre régulier de degré $\ge 3$, et sont donc quasi-isométriques entre eux.

La classification à quasi-isométrie près des TAC de type connexe se ramène à la conjecture classique suivante: si deux groupes de Lie triangulables de la forme $N\rtimes\R$ ($\R$ agissant sur $N$ par contractions) sont quasi-isométriques, alors ils sont isomorphes. De nombreux résultats partiels sont connus, découlant notamment des travaux de Pansu (voir le survol \cite{CFQI}). La thèse d'Hamenst\"adt (1989) affirme résoudre positivement cette conjecture, mais la preuve y est erronée.

Dans le cas mixte un point pas immédiatement évident est de dégager la bonne conjecture. Un premier invariant est le suivant si $G$ est un TAC de type mixte: on considère la fonction modulaire du groupe $G/G_0$; son image est le sous-groupe cyclique de $\Q^*$ engendré par un certain nombre entier $s_G$. Ce nombre n'est pas un invariant QI: en effet si $G=N\rtimes_\sigma\Z$ et $G_k$ est son sous-groupe d'indice fini $N\rtimes_{\sigma^k}\Z$ (d'indice $k$), alors $s_{G_k}=s_G^k$. On définit alors $q_G$ comme le nombre entier $\ge 2$ minimal dont $s_G$ est une puissance entière. On a le théorème récent de Dymarz:


\begin{thm}[Dymarz]
Pour un TAC $G$ de type mixte, le nombre entier $q_G\ge 2$ est un invariant de quasi-isométrie.
\end{thm}

Le nombre $q_G$ apparaît comme le ``pendant totalement discontinu" de $G$. Le ``pendant connexe" apparaît comme le groupe $G/G_\#$, où $G_\#$ est le plus gros sous-groupe distingué localement elliptique de $G$ (localement elliptique voulant dire réunion croissante d'une suite de sous-groupes compacts). De manière intéressante, il existe un invariant numérique qui compare les taux de compaction dans les deux parties.

\begin{defn}Soit $G$ un TAC de type mixte.
On considère les fonctions modulaires de $G/G_0$ et $G/G_\#$; elles induisent par composition des morphismes $\Delta_\#$ et $\Delta_0$ de $G$ vers $\R^*_+$. En notant au préalable que $\Hom(G,\R)$ est de dimension 1, soit $\varpi_G$ l'unique nombre réel tel que $\log\Delta_\#=\varpi_G\log\Delta_0$.
\end{defn}

\begin{thm}[\cite{CQF}]\label{vpi}$\varpi_G$ est un invariant de quasi-isométrie de $G$.
\end{thm}

Ce théorème utilise le théorème \ref{hp0} dans sa preuve.

La conjecture est alors: deux TAC de type mixtes $G,H$ sont quasi-isométriques si et seulement s'ils satisfont:
\begin{itemize}
\item $q_G=q_H$;
\item $\varpi_G=\varpi_H$;
\item les TAC de type mixte $G/G_\#$ et $H/H_\#$ sont quasi-isométriques.
\end{itemize}

Cette conjecture ramène le problème de quasi-isométrie au type connexe. Elle est très fortement motivée par le fait qu'en fait, si la conjecture décrite plus haut pour les groupes triangulables est exacte, alors la conjecture dans le type mixte est vraie. Ceci est établi dans \cite{CQF} en utilisant le théorème de Dymarz et le théorème \ref{vpi}. Notons qu'il ne s'agit pas uniquement de spéculations conjecturales, car dans les cas particuliers où les conjectures sont établies dans les cas de type connexe, il en découle directement un résultat concernant les groupes de type mixte.


\chapter[Propriétés de Kazhdan et Haagerup ]{Travaux autour de la propriété T de Kazhdan et la propriété de Haagerup}\label{chaag}

Cette thématique s'inscrit dans le prolongement de ma thèse. Celle-ci était centrée autour d'une étude de la propriété T relative, notamment dans des réseaux de groupes de Lie connexes et analogues $p$-adiques, typiquement non semi-simples. Les deux principales publications qui en ont alors résulté sont \cite{CJLT,CAENS}. Je décris ci-dessous uniquement les travaux postérieurs à ma thèse.

\section{Propriété T relative}\label{ptr} Les deux principaux travaux que j'ai publié depuis dans ces sujets sont indé\-pen\-dants. Le premier prolonge une note courte du chapitre 7 de ma thèse, un critère pour la propriété T relative dont je ne savais pas prouver la réciproque. Rappelons, si $G$ est un groupe localement compact et $N$ un sous-groupe distingué, que $(G,N)$ a la propriété T relative si dans le dual unitaire de $G$, il existe un voisinage de la représentation triviale constitué de représentations factorisant par $G/N$. Quand $N=G$, cela revient à requérir que la représentation triviale est isolée, et il s'agit de la propriété T telle qu'introduite par Kazhdan \cite{K}. L'exemple prototypique est la propriété T relative pour $(\Gamma\ltimes\R^n,\R^n)$ si $\Gamma$ est un sous-groupe Zariski-dense de $\SL_n(\R)$ pour $n\ge 3$. Burger a observé qu'elle résulte de la non-existence de probabilité invariante sur l'espace projectif de $\R^n$ pour l'action duale de $\Gamma$. Une condition suffisante (mais pas nécessaire, y compris dans le cas scindé) pour la propriété T relative pour $(G,A)$, $A$ étant un sous-groupe abélien distingué fermé de $G$ a été donnée par Shalom \cite{Sha99}: la non-existence d'une moyenne invariante sur les boréliens de $\hat{A}\smallsetminus\{0\}$.

Le résultat principal de \cite{CTETDS} est le suivant:

\begin{thm}
Soit $G=H\ltimes A$ un groupe localement compact $\sigma$-compact avec $A$ abélien. Équivalences:
\begin{enumerate}[(i)]
\item\label{pastr} $(G,A)$ n'a pas la propriété T relative;
\item\label{emoy} il existe une moyenne $m$ dénombrablement approximable sur les boréliens de $\hat{A}$ telle que $m(\{0\})=0$ et $m(V)=1$ pour tout voisinage $V$ de $0$ dans $\hat{A}$.
\item\label{epro} il existe une suite généralisée $(\mu_i)$ de probabilités sur les boréliens de $\hat{A}$ vérifiant les trois conditions suivantes
	\begin{itemize}
	\item $\mu_i$ tend vers $\delta_0$ (convergence faible-* dans $\mathcal{C}_\textnormal{c}(\hat{A})^*)$);
	\item $\mu_i(\{0\})=0$ pour tout $i$;
	\item pour tout $h\in H$ on a $\lim_i\|h\cdot\mu_i-\mu_i\|=0$, et ce uniformément en $h$ sur les compacts de $G$ (autrement dit, $\mu_i$ est asymptotiquement invariante pour la topologie de la norme).
	\end{itemize}
\end{enumerate}
\end{thm}

Ici, une moyenne est dite dénombrablement approximable si elle est dans l'adhérence d'un ensemble dénombrable de probabilités boréliennes. Il est utile d'avoir en tête qu'à l'instar des ultrafiltres, les moyennes sont des objets ésotériques dont l'existence n'est que théorique et permet des énoncés courts. De fait, la \og {vraie\fg} caractérisation de (\ref{pastr}) est (\ref{epro}). L'équivalence entre (\ref{emoy}) et (\ref{epro}), dont la donnée est uniquement un groupe abélien localement compact $A$ muni d'une action d'un groupe $H$, utilise toutefois des arguments astucieux (mais standards) d'analyse fonctionnelle.


L'implication (\ref{pastr})$\Rightarrow$(\ref{epro}) est donnée par le même argument que ses ancêtres, et la principale contribution ici est surtout d'avoir dégagé le bon énoncé. L'implication réellement nouvelle est (\ref{epro})$\Rightarrow$(\ref{pastr}); elle a été prouvée indépendamment et simultanément par Ioana \cite{Ioa}, qui s'est toutefois épargné de significatives difficultés techniques en se cantonnant au cas discret. 

L'article \cite{CTETDS} donne également des caractérisations analogues théoriques, sous les mêmes hypothèses, de la propriété T relative pour
$(A\rtimes H,X)$ quand $X$ est une partie quelconque $A$; elles montrent en particulier que celle-ci ne dépend que de l'image du morphisme $H\to\Aut(A)$ et non de son injectivité ou de la topologie de $H$. Toutefois, il est encore un problème ouvert que d'expliciter ces conditions dans des cas explicites telles qu'un produit en couronne permutationnel $B^{(Y)}\rtimes\Gamma$ ($B$ groupe de type fini abélien, $\Gamma$ groupe discret, $Y$ $\Gamma$-ensemble).

\section{Propriété de Haagerup des produits en couronne}

Un groupe dénombrable $G$ a la {\em propriété de Haagerup} s'il admet une action isométrique métriquement propre sur un espace de Hilbert. Il a la propriété PW (plus forte) s'il admet une action combinatoire (donc isométrique) métriquement propre sur un complexe cubique CAT(0). 

La propriété PW peut être caractérisée comme ceci: une murage sur $G$ est la donnée d'une famille $G$-équivariante de ``murs", à savoir des parties $(D_x)_{x\in X}$ de $G$ indexées par un $G$-ensemble $X$, tel que $D_{gx}=gD_x$ pour tout $(g,x)\in G\times X$, et tel que pour tout $g,h$, l'ensemble des $x$ tel que $\{g,h\}$ rencontre à la fois $D_x$ et son complémentaire, soit fini. Son cardinal $d(g,h)$ est appelé la distance de murs entre $g$ et $h$; c'est une distance invariante à gauche. Le groupe $G$ a la propriété PW s'il admet un murage définissant une distance propre.

Ce concept a une histoire compliquée et a été dégagé sous des points de vues différents et des motivations, et souvent (à tort!) considéré comme moins intéressant que la propriété de Haagerup (elle-même longtemps à l'ombre de la propriété T). Une illustration typique est le résultat principal de \cite{BJS}, qui est présenté comme le fait qu'aucun groupe de Coxeter (infini dénombrable) n'a la propriété T de Kazhdan, quand en fait la preuve consiste à montrer la propriété PW.

Rappelons que le produit en couronne (standard) $W=H\wr G$ est le produit semi-direct $H^{(G)}\rtimes G$, où le groupe $G$ agit par décalage sur le produit direct restreint $H^{(G)}=\bigoplus_{g\in G}H$. Une remarque triviale mais importante est que si $H$ et $G$ sont de type fini, alors $W$ aussi. Du point de vue de la théorie géométrique des groupes, les cas dégénérés sont celui où $H=1$ (si bien que $W=G$) et celui où $G$ est fini (car alors $W$ admet $H^G$ comme sous-groupe d'indice fini). De ce point de vue, le plus ``petit" exemple non trivial est $(\Z/2\Z)\wr\Z$, appelé groupe de l'allumeur de réverbères (lamplighter group) sur $\Z$. 

La question de déterminer si la propriété de Haagerup est stable par produits en couronne était déjà dans l'air quand j'étais en thèse. Le cas connu était précisément celui où $G$ est moyennable, car alors $H^{(G)}$ étant un sous-groupe co-moyennable de $G$ avec la propriété de Haagerup, $W$ hérite de la propriété de Haagerup. Les autres cas étaient ouverts, le cas prototypique étant $(\Z/2\Z)\rtimes F_2$ (où $F_2$ est le groupe libre). 

Indépendamment, il avait été établi qu'un produit en couronne (non dégénéré au sens où $H\neq 1$, $G$ est infini) n'a jamais la propriété T de Kazhdan, et plus précisément que la fonction $\lambda g\mapsto\#\textnormal{Supp}(\lambda)$ est conditionnellement de type négatif ($\lambda\in H^{G}$, $g\in G$). Cela veut dire qu'il existe une action isométrique de $W$ sur un espace de Hilbert (en fait, sur un complexe cubique CAT(0)) propre en restriction à tout sous-ensemble de $H^{(G)}$ sur qui la fonction cardinal du support est propre. Cependant, cette action est bornée en restriction à tout sous-ensemble infini de $H^{(G)}$ constitué de fonctions dont le support est de cardinal borné, et, typiquement dans le cas où $H=\Z/2\Z$, cela rendait la question coriace; en outre, certains résultats analogues de moyennabilité faible ont poussé à douter de la propriété de Haagerup pour $(\Z/2\Z)\wr G$ pour $G$ non moyennable (voir la discussion dans \cite{CSVa}). 

\begin{thm}[{\cite{CSVa,CSVb}}]
La propriété de Haagerup et la propriété PW sont stables par produit en couronne. 
\end{thm}

Pour le cas de la propriété PW, l'argument est entièrement écrit dans \cite{CSVa} dans le cas où $H$ est fini et le cas général s'en déduit par un argument banal, voir \cite{CFW}. Le travail consiste surtout à exhiber la bonne famille de murs équivariants, les vérifications a posteriori sont simples.

Dans le cas de la propriété de Haagerup \cite{CSVb}, il faut travailler avec des familles mesurées de murs (ce qui formellement se ramène à considérer des mesures de Radon invariantes sur l'espace localement compact $2^\Gamma\smallsetminus\{\emptyset,\Gamma\}$ quand $\Gamma$ est un groupe discret (qui peut être $H$, $G$ ou $W$, quoique $H$ joue un rôle mineur). Une partie du travail dans \cite{CSVb} consiste, d'autre part, à adapter la preuve à certains autre types de produits ``permutationnels". Par exemple on montre avec les mêmes arguments:

\begin{thm}[\cite{CSVb}]
Soit $\Gamma$ un groupe et considérons une structure de graphe de Coxeter sur $\Gamma$, invariante à gauche, et soit $C$ le groupe de Coxeter correspondant, sur qui $\Gamma$ agit naturellement en permutant les générateurs canoniques. Si $\Gamma$ a la propriété PW ou de Haagerup, alors le \og produit en couronne de Coxeter\fg  {\em (wreathed Coxeter group)} $W'=C\rtimes\Gamma$ a également cette propriété.
\end{thm}

Notons les deux cas particuliers: si le graphe de Coxeter n'a que des arêtes étiquetées par 2, alors $W'=(\Z/2\Z)\wr\Gamma$ (produit en couronne standard). Si le graphe de Coxeter n'a que des arêtes étiquetées par $\infty$, alors $W'=(\Z/2\Z)\ast\Gamma$ (produit libre).

Comme évoqué à la fin du \S\ref{ptr}, la question de déterminer quels produits en couronne permutationnels $B\wr_Y\Gamma$ ont la propriété de Haagerup est largement ouverte, y compris dans le cas où $B$ est un groupe à deux éléments et $\Gamma$ est un groupe libre.

\chapter{Topologie de Chabauty}

Le terme ``topologie de Chabauty" désigne plusieurs concepts apparentés. Chabauty a introduit une topologie, compacte, sur l'ensemble des sous-groupes fermés de $\R^d$, qui est une compactification naturelle de l'ensemble des réseaux. La définition s'étend en fait à l'ensemble $\mathcal{F}(X)$ des fermés d'un espace topologique localement compact quelconque $X$. Cette topologie, compacte, est caractérisée par le fait que $\lim F_i=F$ si et seulement si, en notant $\mathcal{V}(x)$ l'ensemble des voisinages de $x\in X$
\[\left\{ \begin{array}{r}
\forall x\in F,\; \exists V\in\mathcal{V}(x),\; \exists i_0\;\forall i\ge i_0,\; F_i\cap V\neq\emptyset; \\
\forall x\notin F,\; \exists V\in\mathcal{V}(x),\; \exists i'_0\;\forall i\ge i'_0,\; F_i\cap V=\emptyset.
\end{array}\right.\]
Si $G$ est un groupe localement compact, l'ensemble $\mathcal{S}(G)$ des sous-groupes fermés (respectivement l'ensemble $\mathcal{N}(G)$ de ses sous-groupes distingués fermés) de $G$ est fermé dans $\mathcal{F}(G)$ et est ainsi muni d'une topologie compacte naturelle. 

Dans le cas où $X$ est discret, la topologie sur $\mathcal{F}(X)=2^X$ n'est autre que la topologie produit. Bien qu'il serait probablement plus adapté de l'attribuer à Cantor ou Tychonoff, l'usage est cependant de l'appeler également topologie de Chabauty.

Une découverte de Grigorchuk \cite{Gri84} est que la topologie de Chabauty sur l'ensemble des sous-groupes distingués d'un groupe libre $F_n$ sur $n$ générateurs peut s'interpréter comme espace $\mathcal{G}_n$ des groupes marqués (en interprétant le point $N$ comme le groupe $F/N$ marqué par le $n$-uplet image des générateurs canoniques). Cela permet de définir, au coût de ce marquage, une topologie agréable sur l'espace des groupes de type fini. Une tentative dans ce sens avait été effectuée par Gromov peu avant \cite{Gro81}, mais correspond plutôt au quotient de cet espace par la relation d'isomorphisme de groupes (non marqués), qui n'est pas un espace séparé et n'est donc pas le bon objet.



Une manière de répartir mes travaux sur ce sujet est la suivante:

\begin{itemize}
\item[*] Groupes marqués:
	\begin{itemize}
	\item \cite{CGP} (avec L.~Guyot et W.~Pitsch) {\em On the isolated points in 	the space of groups}. J.~Algebra, 24 pages (2007);
	\item \cite{YMA} {\em A sofic group away from amenable groups}. Math.~Ann., 7 	pages (2011);
	\item \cite{YAF} {\em On the Cantor-Bendixson rank of metabelian groups}. Ann.\ Inst.\ Fourier, 20 pages (2011). 
	\item \cite{BCGS} (avec R.~Bieri, L.~Guyot et R.~Strebel) {\em Infinite 	presentability of groups and condensation} (à paraître au Journal de l'IMJ).
	\end{itemize}
\item[*] Espace de sous-groupes, idéaux, sous-modules:
	\begin{itemize}
	\item \cite{YIJAC} {\em The space of finitely generated rings}. 10 pages 	(2009).
	\item \cite{CGP2}(avec L.~Guyot et W.~Pitsch) {\em The space of subgroups of an abelian group}. J.~London Math.\ Soc., 30 pages (2010).
	\item \cite{YAGT} {\em On the Chabauty space of locally compact abelian groups},  Algebr. Geom. Topol., 29 pages (2011).
	\end{itemize}
\end{itemize}

Le point de départ est l'article \cite{CGP}, écrit quand j'étais en thèse mais non inclus dans celle-ci. Il contient l'observation inédite \cite[lemme 1.3(3)]{CGP} que la topologie (et l'ordre) local autour d'un groupe marqué dans $\mathcal{G}_n$ ne dépend que de la structure du groupe sous-jacent. En particulier, toute propriété topologique locale en un point (être isolé, posséder un voisinage dénombrable\dots) est une propriété intrinsèque d'un groupe de type fini. L'article \cite{CGP} contient diverses caractérisations et propriétés des groupes isolés, ainsi que de nombreux exemples (les seuls observés jusqu'alors étaient, à peu de choses près, les groupes finis et les groupes de présentation finie simples); il démontre également le fait surprenant que l'ensemble des points isolés n'est pas dense.

Le court article \cite{YMA} construit un groupe sofique isolé non moyennable. La méthode n'introduit pas de nouvelle méthode de soficité, le groupe obtenu étant (localement résiduellement fini)-par-moyennable et la soficité de tels groupes étaient bien connue, cependant cela répondait à une question posée par plusieurs auteurs depuis une dizaine d'années \cite{AlGG,Pe,Th}, à savoir trouver un groupe sofique non limite de groupes moyennables. Le groupe est en fait quotient d'un groupe linéaire par un sous-groupe distingué abélien; ceci est un ingrédient utile pour cuisiner des (contre\nobreakdash-)exemples surprenants en théorie des groupes, déjà utilisé par Abels dans \cite{Ab1} et convenablement recyclé dans ma thèse \cite{YPAMS}, ainsi que dans \cite{BCGS,CTAB1}.

L'article \cite{YIJAC} était initialement conçu comme préliminaire à \cite{YAF}; outre alléger, le détacher m'a permis d'écrire mon premier article sans le mot \og groupe\fg! Il s'agit dans \cite{YIJAC} de considérer l'espace des anneaux commutatifs marqués à $n$ générateurs et de considérer sa topologie, et plus précisément de décrire la topologie au voisinage d'un anneau commutatif de type fini quelconque. Plus généralement, il décrit la topologie de l'espace des sous-modules d'un module de type fini sur un tel anneau.

Ce résultat peut être intéressant en soi; je l'ai considéré en étant initialement motivé par le fait que si on a un groupe métabélien de type fini $G$, et si $M$ est un sous-groupe abélien distingué tel que $G/M$ est abélien, alors les sous-groupes distingués de $G$ inclus dans $M$ sont précisément les sous-modules de $M$ (qu'on voit comme module sur l'anneau ---\,commutatif de type fini\,--- engendré par $G/M$). Ceci permet notamment de montrer \cite{YAF} que l'espace des groupes marqués à $d\ge 2$ générateurs contient des ouverts fermés dénombrables de rang de Cantor-Bendixson $\alpha$ pour tout $\alpha<\omega^\omega$. 

L'article \cite{BCGS} s'intéresse à divers renforcements de la propriété d'être de présen\-tation infinie, pour un groupe de type fini. Ces renforcements sont notamment ($S$ désigne une partie génératrice finie):
\begin{enumerate}[(1)]
\item\label{i_pm} l'existence d'une présentation infinie minimale $\langle S\mid r_n, n\ge 0\rangle$, où $r_n$ n'est, pour aucun $n$, conséquence des $r_m$ pour $m\neq n$;
\item l'existence d'une présentation infinie indépendante $\langle S\mid R_n,n\ge 0\rangle$, où $R_n$ est un ensemble (éventuellement infini) de relateurs, qui n'est pour aucun $n$ conséquence de $\bigcup_{m\neq n}R_m$;
\item être de condensation extrinsèque, au sens où pour tout groupe de présen\-tation finie $H$ et sous-groupe distingué $N$ de $H$ tel que $G$ est isomorphe à $G/N$, l'ensemble des sous-groupes distingués de $H$ inclus dans $N$ est non dénom\-brable (cela signifie que comme groupe marqué, $G$ est un point de condensation parmi les groupes au-dessus de $G$, c'est-à-dire ayant $G$ comme quotient marqué);
\item\label{i_pi} être de présentation infinie.
\end{enumerate}

Chacune de ces propriétés implique la suivante et des contre-exemples à chacune des implications sont donnés dans \cite{BCGS}; ceci répond de plusieurs manières à la question initiale, posée par Grigorchuk à Luc Guyot, qui était de savoir si (\ref{i_pi}) implique (\ref{i_pm}).

L'autre principal résultat de \cite{BCGS} est le suivant:

\begin{thm}
Tout groupe métabélien de type fini et de présentation infinie est de condensation extrinsèque. En particulier, il est de condensation. 
\end{thm}

Ici, être de condensation signifie que tout voisinage dans l'espace des groupes marqués est non dénombrable. Réciproquement, pour un groupe métabélien de présen\-tation finie, il existe un voisinage dénombrable, le calcul de son rang de Cantor-Bendixson est effectué dans \cite{YAF}.

Les papiers \cite{CGP2,YAGT} sont moins directement motivés par l'étude de l'espace des groupes marqués. Le premier \cite{CGP2} décrit précisément la topologie de l'espace des groupes d'un groupe dénombrable discret quelconque. Le second donne des résultats moins précis mais dans le cadre plus compliqué des groupes abéliens localement compacts, où les espaces obtenus ne sont pas nécessairement totalement discontinus: calcul de la dimension topologique, caractérisation de la connexité, etc.

Plus précisément: dans \cite{CGP2}, on considère la topologie de l'espace $\mathcal{S}(A)$ des sous-groupes d'un groupe abélien discret quelconque $A$. 
Il est toujours compact et totalement discontinu. Si $A$ est non dénombrable, c'est un espace parfait. Si $A$ est dénombrable, on caractérise sa topologie, et plus précisément étant donné un sous-groupe $B$ de $A$, on décrit la topologie de $\mathcal{S}(A)$ au voisinage de $B$.

Ces espaces sont relativement petits, en terme d'analyse de Cantor-Bendixson. Si $A$ n'est pas minimax (un groupe abélien est dit minimax s'il est isomorphe à un sous-quotient de $\Z[1/k]^n$ pour des entiers $k\ge 1$, $n\ge 0$), l'espace $\mathcal{S}(A)$ est parfait, et est donc un espace de Cantor si $A$ est dénombrable. Le cas intéressant est donc celui où $A$ est minimax. Le rang de Cantor-Bendixson est toujours fini (en contraste avec les exemples de \cite{YIJAC,YAF}), et $\mathcal{S}(A)$ a une topologie parmi les suivantes:

\begin{thm}
Soit $A$ un groupe abélien discret minimax. Alors $\mathcal{S}(A)$ est homéo\-morphe à l'un des espaces suivants (définis ci-dessous):
\begin{itemize}
\item $D^n\times [m]$ pour un certain $n\in\N$, $m\in\N^*$;
\item $D^n\times W$.
\end{itemize}
\end{thm}

Ce théorème est extrait des théorèmes C et D de \cite{CGP2}, qui sont plus précis. Ici $[m]$ est un ensemble fini discret à $m$ éléments; $D$ est l'union d'une suite convergente et de sa limite, et $W$ est l'``espace de Cantor poussiéreux": c'est un espace métrisable compact totalement discontinu, réunion d'un fermé homéomorphe à un espace de Cantor et d'une partie ouverte, discrète, dénombrable et dense; ces propriétés caractérisent $W$ à homéomorphisme près (c'est un cas particulier de \cite[Proposition 1.3.2]{CGP2}).

Dans \cite{YAGT} je me suis intéressé à la topologie de l'espace $\mathcal{S}(G)$ quand $G$ est un groupe abélien localement compact. Contrairement au cas où $G$ est discret, cet espace n'est pas forcément totalement discontinu. Par exemple, $\mathcal{S}(\R^n)$ admet l'ensemble des réseaux de $\R^n$ comme ouvert dense, qui est une variété topologique de dimension $n^2$. On voit aisément que $\mathcal{S}(\R)$ est homéomorphe à un segment; Hubbard et Pourezza ont montré le résultat surprenant que $\mathcal{S}(\R^2)$ est homéomorphe à une sphère de dimension 4. On peut vérifier que le point $\Z^{n-1}\times\{0\}$ de $\mathcal{S}(\R^n)$ possède un voisinage homéomorphe au produit de $\R^{n^2-n}$ et du cône sur le tore de dimension $n-1$. Il s'ensuit que pour $n\ge 3$, l'espace $\mathcal{S}(\R^n)$ n'est pas une variété topologique. Kloeckner montre cependant qu'il admet une stratification de Goresky-MacPherson; une conséquence notable est que cet espace est localement contractile.

Si $G$ est un groupe abélien localement compact, soit $G^\vee=\Hom(G,\R/\Z)$ son dual de Pontryagin. En outre, il existe un unique entier $r=r(G)\in\N$ tel que $G$ est isomorphe à un produit $\R^k\times H$ et $H$ possède un sous-groupe compact ouvert. Un échantillon des résultats obtenus dans \cite{YAGT} est le suivant:

\begin{thm}[\cite{YAGT}]Soit $G$ un groupe abélien localement compact.
\begin{enumerate}
\item\label{oc} L'application ``orthogonal" $\mathcal{S}(G)\to\mathcal{S}(G^\vee)$ est un homéomorphisme.
\item La dimension topologique de $\mathcal{S}(G)$ est égale à $\dim(G)\dim(G^\vee)$, où $0\cdot\infty=\infty\cdot 0=0$. En particulier, $\mathcal{S}(G)$ est totalement discontinu si et seulement si $G$ est elliptique (i.e.\ réunion croissante de sous-groupe compacts) ou totalement discontinu.
\item Si $r(G)\ge 1$, alors $\mathcal{S}(G)$ est connexe.
\item Si $r(G)=0$, alors les composantes connexes de $\mathcal{S}(G)$ sont toutes homéo\-morphes à des groupes compacts; le nombre de composantes connexes de $\mathcal{S}(G)$ est infini à moins que $G$ ne soit fini.
\end{enumerate}
\end{thm}

Ici (\ref{oc}) a été affirmé par Protasov \cite{PTd}, mais sa preuve est seulement une réduction au cas non trivial de $G=\R^n$.

\appendix\chapter{Quelques problèmes ouverts}

\begin{enumerate}
\item Soit $\K$ un corps localement compact non discret de caractéristique non nulle. Soit $\mathbb{G}$ un $\K$-groupe algébrique affine absolument simple de type adjoint, de $\K$-rang supérieur ou égal à 2, et soit $G=\mathbb{G}(\K)$. Soit $\hat{G}$ le groupe des automorphismes du groupe topologique $G$; il contient le groupe $G$ comme sous-groupe fermé cocompact distingué. On dira qu'un réseau $\Gamma$ de $\hat{G}$ est {\em standard} si $\Gamma\cap G$ est un réseau dans $G$. {\bf Est-ce que $\hat{G}$ possède des réseaux non standards?} si oui, peuvent-ils être classifiés? Les réseaux standards sont classifiés, à commensurabilité près, par le théorème d'arithméticité; en particulier, dans certains cas $G$ n'admet aucun réseau cocompact, mais on ne sait pas si $\hat{G}$ en admet.

\item On sait que le groupe de K-théorie $K_2(\Z)$ est cyclique d'ordre 2. Soit $^\omega\Z$ l'ultraproduit de $\Z$ par rapport à un ultrafiltre non principal $\omega$ sur les entiers. L'injection évidente $K_2(\Z)\to K_2(^\omega\Z)$ est-elle un isomorphisme?

\item (cf.\ la fin du chapitre \ref{chafon}) On sait que la fonction de Dehn d'un groupe de Lie simplement connexe 2-nilpotent est entre quadratique et cubique. Peut-on la déterminer (en fonction de propriétés algébriques de l'algèbre de Lie)? en particulier, peut-on déterminer quand elle est quadratique, respectivement cubique?

\item (cf. la fin du chapitre \ref{chaag}) Soit $F$ un groupe libre de type fini, $H$ un sous-groupe. Soit $C$ un groupe cyclique non trivial. Quand-est ce que le produit en couronne $C\wr_{F/H}F$ a la propriété de Haagerup? (La réponse dépend de $H$, mais probablement pas de $C$. Lorsque $H$ est distingué, la réponse est oui si et seulement si $F/H$ a la propriété de Haagerup, par \cite{CSVb,CIo}.)

\end{enumerate}



\begin{thebibliography}{KM98b}


\bibitem[]{} {\large\bf Chapitre 1}

\bibitem{CJT} Y.~\balank{1.5cm}. {\it Dimension of asymptotic cones of Lie groups}. {\bf J. Topology} 1(2), 343--361, 2008.

\bibitem{CI} Y.~\balank{1.5cm}. {\it Asymptotic cones of Lie groups and cone equivalences}. {\bf Illinois J. Math.} 55(1) (2011), 237--259.



\bibitem[]{} {\large\bf Chapitre 2}

\bibitem{CTcm} Y.~\balank{1.5cm}, R. Tessera. {\em Metabelian groups with quadratic Dehn function and Baumslag-Solitar groups}. {\bf Confluentes Math.} 2(4) (2010) 431--443.

\bibitem{CTAB1} Y.~\balank{1.5cm}, R. Tessera. {\em Dehn function and asymptotic cones of Abels' group}. {\bf J. Topology} (2013) 6 (4): 982--1008.

\bibitem{CTge} Y.~\balank{1.5cm}, R. Tessera. {\em Geometric presentations of Lie groups and their Dehn functions}. Prépublication, 2013. [Publ. Math. IHES 125(1) (2017), 79-219.]

\bibitem[]{} {\large\bf Chapitre 3}

\bibitem{CAM} Y.~\balank{1.5cm}, R.~Tessera. Contracting automorphisms and $L^p$-cohomology in degree one.  {\bf Ark. Mat.} 49(2) (2011) 295--324.

\bibitem{CCMT} Y.~\balank{1.5cm}, P-E. Caprace, N. Monod, R. Tessera. Amenable hyperbolic groups (40 pages), à paraître à Journal European Math. Soc. [J. Eur. Math. Soc. 17 (2015), 2903-2947]

\bibitem{CQF} Y.~\balank{1.5cm}. {\em Commability and focal locally compact groups}. Prépublication, à paraître à {\bf Indiana Univ. Math. J.}, 2013. [Indiana Univ. Math. J. 64(1) (2015) 115-150]

\bibitem{CFQI} Y.~\balank{1.5cm}. On the quasi-isometric classification of focal hyperbolic groups. Brouillon, 2012. [Pages 275-342 in: New directions in locally compact groups (P-E. Caprace, N. Monod edts), London Math. Soc. Lecture Notes Series 447, Cambridge University Press 2018]

\bibitem[]{} {\large\bf Chapitre 4}

\bibitem{CJLT} Y. Cornulier. Kazhdan and Haagerup Properties in algebraic groups over local fields. {\bf J. Lie Theory} 16, 67--82, 2006.

\bibitem{CAENS} Y.~\balank{1.5cm}. {\em Relative Kazhdan Property}. {\bf Annales Sci. ENS} 39 (2006), no. 2, 301-333.

\bibitem{CTETDS} Y.~\balank{1.5cm} \and R. Tessera. {\em A characterization of relative Kazhdan property T for semidirect products with abelian groups}. {\bf Ergod. Th. \& Dynam. Sys.} (2011), 31, 793--805.

\bibitem{CSVa} Y.~\balank{1.5cm}, Y. Stalder, A. Valette. Proper actions of lamplighter groups associated with free groups. {\bf CRAS Paris, Ser. I} 346 (2008), no. 3-4, 173--176.

\bibitem{CSVb} Y.~\balank{1.5cm}, Y. Stalder, A. Valette. {\em Proper actions of wreath products and generalizations.} {\bf Trans. Amer. Math. Soc.} 364(6) (2012) 3159--3184.

\bibitem{CFW} Y.~\balank{1.5cm}. Group actions with commensurated subsets, wallings and cubings. Draft, 2013. [Unpublished: arXiv:1302.5982]

\bibitem[]{} {\large\bf Chapitre 5}



\bibitem{YPAMS} Y.~\balank{1.5cm}. {\em Finitely presentable, non-Hopfian groups with Kazhdan's Property and infinite outer automorphism group}. {\bf Proc. Amer. Math. Soc.} 135, 951--959, 2007.

\bibitem{YMA} Y.~\balank{1.5cm}. {\em A sofic group away from amenable groups}. Math. Ann. 350(2) (2011) 269-275.

\bibitem{CGP}
Y.~\balank{1.5cm}, L.~Guyot, and W.~Pitsch.
\newblock On the isolated points in the space of groups.
J. Algebra 307 (2007), no. 1, 254--277.

\bibitem{BCGS} R. Bieri, Y.~\balank{1.5cm}, L.~Guyot,  R.~Strebel. {\em Infinite 	presentability of groups and condensation}, prépublication 2012, à paraître à {\bf J. Inst. Math. Jussieu}. [J. Inst. Math. Jussieu 13(4) (2014) 811-848.]

\bibitem{YIJAC} Y.~\balank{1.5cm}. The space of finitely generated rings {\bf Internat. J. Algebra Comput.} 19(3) (2009) 373-382.

\bibitem{YAF} Y.~\balank{1.5cm}. {\em On the Cantor-Bendixson rank of metabelian groups}. {\bf Ann. Inst. Fourier} 61(2) (2011), 593-618.

\bibitem{CGP2} Y.~\balank{1.5cm}, L. Guyot and W. Pitsch.. The space of subgroups of an abelian group. {\bf J. London Math. Soc.} 81(3) (2010) 727-746.

\bibitem{YAGT} Y.~\balank{1.5cm}. On the Chabauty space of locally compact abelian groups. {\bf Algebr. Geom. Topol.} 11 (2011) 2007--2035.





\bibitem[Ab79]{Ab1} H. Abels.
\newblock An example of a finitely presented solvable group.  Homological group theory. 
\newblock {\em (Proc. Sympos., Durham, 1977),  pp. 205--211, London Math. Soc. Lecture Note Ser., 36, Cambridge Univ. Press, Cambridge-New York,} 1979.


\bibitem[AlGG01]{AlGG} M.A. Alekseev, L. Yu. Glebskii, E.I. Gordon. On approximations of groups, group actions and Hopf algebras. (Russian) Zap. Nauchn. Sem. S.-Peterburg. Otdel. Mat. Inst. Steklov. (POMI) 256 (1999), Teor. Predst. Din. Sist. Komb. i Algoritm. Metody. 3, 224--262, 268; translation in J. Math. Sci. (New York) 107 (2001), no. 5, 4305--4332.

\bibitem[AO02]{AO} G. Arzhantseva, D. Osin. Solvable groups with polynomial Dehn functions. Trans. Amer. Math. Soc. 354 (2002), no. 8, 3329--3348. 

\bibitem[AW76]{AW} R. Azencott and E. N. Wilson, Homogeneous manifolds with negative curvature. I, Trans. Amer. Math. Soc. 215 (1976), 323-362.

\bibitem[BJS88]{BJS} M. Bozejko, T. Januszkiewicz, R. Spatzier. Infinite Coxeter groups do not have Kazhdan's property. J. Operator Theory 19 (1988), no. 1, 63--67. 

\bibitem[BORS02]{BORS} J-C. Birget, A. Yu. Olshanskii, E. Rips, M. Sapir. Isoperimetric functions of groups and computational complexity of the word problem. Ann. of Math. (2) 156 (2002), no. 2, 467--518.

\bibitem[BS78]{BiS} R. Bieri, R. Strebel. Almost finitely presented soluble groups. Comment. Math. Helv., 53(2):258--278, 1978.

\bibitem[BSSW11]{BSSW} A. Bendikov, L. Saloff-Coste, M. Salvatori, \and W. Woess. The heat semigroup and Brownian
motionon strip complexes. Advances in Math. 226 (2011) 992--1055.

\bibitem[Bur99]{Bur} J. Burillo. Dimension and fundamental groups of asymptotic cones. J. London Math. Soc. (2) 59 (1999), no. 2, 557--572. 

\bibitem[Cha50]{Cha} C. Chabauty.
\newblock Limite d'ensembles et g\'eom\'etrie des nombres.
\newblock Bull. Soc. Math. France 78 (1950), 143--151. 


\bibitem[CI11]{CIo} I. Chifan, A. Ioana. On relative Property (T) and Haagerup's Property. Trans. Amer. Math. Soc 363 (2011), 6407--6420.

\bibitem[DW84]{DW} L. van den Dries, A. Wilkie. Gromov's Theorem on Groups of Polynomial Growth and Elementary Logic. J. Algebra 89(2), 1984, 349--374.

\bibitem[EP01]{EP} A. Dyubina (Erschler), I. Polterovich. Explicit constructions of universal $\mathbf{R}$-trees and asymptotic geometry
of hyperbolic spaces. Bull. London Math. Soc. 33 (2001), no. 6, 727--734.


\bibitem[Gri84]{Gri84}
R.~Grigorchuk.
\newblock Degrees of growth of finitely generated groups and the theory of
  invariant means.
\newblock {\em Izv. Akad. Nauk SSSR Ser. Mat.}, 48(5):939--985, 1984.

\bibitem[Gro81]{Gro81}
M.~Gromov.
\newblock Groups of polynomial growth and expanding maps.
\newblock {\em Inst. Hautes Etudes Sci. Publ. Math.}, 53 (1981) 53--73.

\bibitem[Gro87]{Gro87} M. Gromov Hyperbolic groups, Essays in group theory, Math. Sci. Res. Inst. Publ., vol. 8, Springer, New York, 1987 75--263.

\bibitem[Gro93]{Gro93} M. Gromov. Asymptotic invariants of
infinite groups. Geometric Group Theory, London Math. Soc. Lecture
Note Ser. (G. Niblo and M. Roller, eds.), no. 182, 1993.

\bibitem[HP79]{HP} J. Hubbard, I. Pourezza. The space of closed subgroups of $\mathbf{R}^2$, dans Topology 18 (1979), no 2, 143--146.

\bibitem[Ioa10]{Ioa} A. Ioana. Relative property (T) for the subequivalence relations induced by the action of SL2(Z) on T2. Adv. Math 224 (2010), 1589--1617

\bibitem[Kaz67]{K} D. Kazhdan. \newblock On the connection of the dual
of a group with the structure of its closed subgroups. Funct.   
Anal. Appl. {\bf 1}, 63--65, 1967.

\bibitem[Kl09]{Klo} B. Kloeckner. The space of closed subgroups of $\mathbf{R}^n$ is stratified and simply connected. J.~Topol. 2(3) (2009) 570--588.

\bibitem[KSTT05]{KSTT} L. Kramer, S. Shelah, K. Tent, S. Thomas. Asymptotic cones of finitely presented groups. Adv. Math. 193 (2005), no. 1, 142--173.

\bibitem[MNO92]{MNO} C. Mayer, J. Nikiel and L. G. Oversteegen. On universal $\mathbf{R}$-trees. Trans. Amer. Math. Soc. 334 (1992), 411--432.

\bibitem[Pan83]{Pan1} P. Pansu. {\em Croissance
des boules et des g\'eod\'esiques ferm\'ees dans les
nilvari\'et\'es}. Ergodic Theory Dyn. Syst. {\bf 3}, 415-445, 1983.

\bibitem[Pan07]{Pan07} P. Pansu. Cohomologie $L^p$ en degré 1 des espaces homogènes, Potential Anal. 27 (2007), 151--165.

\bibitem[Pau01]{Pauls}
S. D. Pauls. \newblock {\em The large scale geometry in  
nilpotent Lie groups}. \newblock Commun. Anal. Geom. {\bf 9}(5),
951-982, 2001.

\bibitem[Pe08]{Pe} V. Pestov.
{\em Hyperlinear and sofic groups: a brief guide}. 
Bull. Symbolic Logic 14(4) (2008) 449--480. 

\bibitem[Pr79]{PTd} I. Protasov.
Dualisms of topological abelian groups. (Russian) 
Ukrain. Mat. Zh. 31 (1979), no. 2, 207--211, 224. 
English translation: Ukrainian Math. J. 31 (1979), no. 2, 164--166.

\bibitem[Sha99]{Sha99} Yehuda Shalom. \newblock
Invariant measures for algebraic actions, Zariski dense
subgroups and Kazhdan's property~(T). \newblock Trans. Amer.
Math. Soc. \textbf{351}, 3387-3412, 1999.

\bibitem[Tes09]{T} R. Tessera. Vanishing of the first reduced cohomology with values in an Lp-
representation, Ann. Inst. Fourier (Grenoble) 59 (2009), 851--876.

\bibitem[Th10]{Th} A. Thom. {\em Examples of hyperlinear groups without factorization property}. Groups Geom. Dyn. 4(1) (2010) 195--208.

\bibitem[TV00]{TV} S. Thomas, B. Velickovic. Asymptotic cones of finitely generated groups. Bull. London Math. Soc. 32 (2000), no. 2, 203--208.

\bibitem[Wen11]{Wen} S. Wenger. Nilpotent groups without exactly polynomial Dehn function. J. Topology 4 (2011), 141--160.







\end{thebibliography}
\end{document}